\documentclass{amsart}%
\usepackage{amsfonts}
\usepackage[margin=1.46in]{geometry}
\usepackage{hyperref}
\usepackage[all,cmtip]{xy}
\usepackage{diagxy}
\usepackage{amsmath}
\usepackage{amssymb}

\usepackage{graphicx}
\newtheorem{theorem}{Theorem}
\newtheorem*{maintheorem}{Main Theorem}
\newtheorem*{slogan}{Slogan}
\numberwithin{theorem}{section}
\theoremstyle{plain}

\newtheorem{corollary}[theorem]{Corollary}

\newtheorem{definition}[theorem]{Definition}

\newtheorem{lemma}[theorem]{Lemma}

\newtheorem{proposition}[theorem]{Proposition}

\theoremstyle{remark}

\newtheorem{example}[theorem]{Example}
\newtheorem{problem}{Question}
\newtheorem{remark}[theorem]{Remark}
\newtheorem{history}[theorem]{History}
\newtheorem{elab}[theorem]{Elaboration}
\numberwithin{equation}{section}
\begin{document}
\title[Volume of line bundles via valuation vectors]{Volume of line bundles via valuation vectors\linebreak(different from Okounkov bodies)}
\author{O. Braunling}
\address{Mathematical Institute, University of Freiburg, Ernst-Zermelo-Strasse 1, 79104
Freiburg im Breisgau, Germany}
\thanks{The author was partially supported by DFG GK1821 \textquotedblleft
Cohomological Methods in Geometry\textquotedblright.}

\begin{abstract}
Up to a factor $1/n!$, the volume of a big line bundle agrees with the
Euclidean volume of its Okounkov body. The latter is the convex hull of top
rank valuation vectors of sections, all with respect to a single flag. In this
text we give a different volume formula, valid in the ample cone, also based
on top rank valuation vectors, but mixing data along several different flags.

\end{abstract}
\maketitle

A recent paper of K\"{u}ronya and Lozovanu uses an interesting viewpoint to
study Newton--Okounkov bodies $\Delta_{\underline{v}}(D)$ of surfaces: The
valuation vectors whose convex hull forms $\Delta_{\underline{v}}(D)$ can be
interpreted as local intersection numbers, \cite[Remark 1.8]{MR3781431}. This
idea is the starting point for this text. A priori unrelated to this, work of
Parshin expresses \textit{global} intersection numbers in terms of rank $n$
valuation vectors \cite{MR697316}. In this paper we combine both ideas. We get
a new formula for the volume of $D$, using valuation vectors belonging to a
number of \textit{different} valuations/flags. As such, it is of a quite
different nature than the construction of Okounkov bodies, which uses only a
single valuation. Pursuing this idea, we are led to a formula whose summands
all have the shape \textquotedblleft$\frac{1}{n!}\det(\ldots)$%
\textquotedblright, which has a very tempting interpretation as the volume of
an $n$-simplex. Although it remains an unresolved mystery, this suggests that
our formula might also have an interpretation in terms of convex polytopes $-$
albeit necessarily of a different nature than Okounkov bodies.\medskip

Let us explain this a little slower and more carefully: If $\underline{x}%
^{1},\ldots,\underline{x}^{n}\in\mathbb{R}^{n}$ are vectors, then by%
\begin{equation}
\left.  \mathrm{simplex}\left\langle \underline{x}^{1},\underline{x}%
^{2},\ldots,\underline{x}^{n}\right\rangle \right.  \label{luma9a}%
\end{equation}
we refer to the oriented simplex formed as the convex hull of $\{0,\underline
{x}^{1},\underline{x}^{2},\ldots,\underline{x}^{n}\}$. We also remember the
orientation, i.e. whether $\underline{x}^{1},\underline{x}^{2},\ldots
,\underline{x}^{n}$ define the standard orientation or not. It has a (signed)
volume,%
\begin{equation}
\operatorname*{Vol}\left.  \mathrm{simplex}\left\langle \underline{x}%
^{1},\underline{x}^{2},\ldots,\underline{x}^{n}\right\rangle \right.
=\frac{1}{n!}\det\left[  \,\left.  \underline{x}^{i}\right.  \,\right]
_{i=1,\ldots,n}\text{.} \label{luma9}%
\end{equation}
Any divisor $D$ on an irreducible $n$-dimensional smooth projective variety
$X/k$ has an invariant called its `\textit{volume}'. This invariant is usually
defined in terms of the growth of global sections under taking powers of the
attached line bundle. Over $\mathbb{C}$, and if $D$ happens to be very ample,
one can alternatively pull back the Fubini--Study metric of $\mathbb{P}%
_{\mathbb{C}}^{N}$ along the associated projective embedding%
\[
X\hookrightarrow\mathbb{P}_{\mathbb{C}}^{N}%
\]
and then the Riemannian metric volume of $X$ can also be taken as the
definition, at least after rescaling it by the factor $\frac{1}{n!}$. However,
the concept of `\textit{volume}' can also be understood as the real volume of
a convex body in Euclidean space. To do this, pick a top rank valuation
$\underline{v}:k\left(  X\right)  ^{\times}\rightarrow\mathbb{Z}%
_{\operatorname*{lex}}^{n}$ on the function field of $X$, and define%
\[
\Delta_{\underline{v}}(D):=\left.  \mathrm{closed}\text{ }\mathrm{convex}%
\text{ }\mathrm{hull}\right.  \left\{  \frac{1}{m}\underline{v}(s)\right\}
_{s,m}\qquad\text{inside }\mathbb{R}^{n}\text{,}%
\]
where $(s,m)$ runs through all pairs $m\geq1$ and $s\in H^{0}(X,\mathcal{O}%
_{X}(mD))\setminus\{0\}$. This is the so-called \emph{Newton--Okounkov body}
of $D$ for the valuation $\underline{v}$. Changing the valuation gives
different convex bodies $\Delta_{(-)}(D)$, but they always have the same
volume. This volume is again $\frac{1}{n!}$-th of the birational
`\textit{volume}'. This characterization of volume is due to Lazarsfeld and
Musta\c{t}\u{a} \cite{MR2571958}.

Let us focus on a rather special situation:\ Assume that $D$ is ample and that
the so-called graded semigroup $\Gamma_{\underline{v}}(D)$ is finitely
generated (we define this carefully later). Then we can phrase the above fact
in a stronger form: For $m\gg1$ big enough to make $mD$ very ample, we get a
local trivialization $(U_{\alpha},h_{\alpha})_{\alpha\in I}$ of $mD$ as a
Cartier divisor on a finite open cover $(U_{\alpha})_{\alpha\in I}$; each
$h_{\alpha}$ comes from a global section of $\mathcal{O}_{X}(mD)$, and
moreover%
\[
\Delta_{\underline{v}}(D)=\mathrm{convex}\text{ }\mathrm{hull}\left\{
\frac{1}{m}\underline{v}(h_{\alpha})\right\}  _{\alpha\in I}\text{.}%
\]
In particular, in this case $\Delta_{\underline{v}}(D)$ is a convex polytope.
By a `local trivialization as a Cartier divisor' we mean an explicit \v{C}ech
$0$-cocycle $\check{H}^{0}(\{U_{\alpha}\},\mathcal{K}_{X}^{\times}%
/\mathcal{O}_{X}^{\times})$. Note that we take the convex hull running over
local equations $h_{\alpha}$, but all with respect to the same valuation. Our
main result gives a different formula for the volume of $D$. It\newline%
$\left.  \qquad\right.  $(A) also uses only the global sections $h_{\alpha}$
as above, but\newline$\left.  \qquad\right.  $(B) with respect to a family of
different top rank valuations $\underline{w}$,\newline depending on the global
geometry of the variety.

\begin{maintheorem}
Let $X/k$ be an irreducible smooth projective variety of dimension $n$.

\begin{enumerate}
\item Let $\underline{v}:k\left(  X\right)  ^{\times}\rightarrow
\mathbb{Z}_{\operatorname*{lex}}^{n}$ be a top rank valuation.

\item Let $D$ be an ample divisor.

\item Suppose the graded semigroup $\Gamma_{\underline{v}}(D)$ is finitely generated.

\item Choose $m\geq1$ such that $mD$ is very ample.
\end{enumerate}

Then there is a local trivialization of $mD$ as a Cartier divisor,
$(U_{\alpha},h_{\alpha})_{\alpha\in I}$ such that the $h_{\alpha}$ are
restrictions of global sections to the opens $U_{\alpha}$. Moreover,%
\begin{align}
&  \operatorname*{Vol}\left(  \underset{\alpha\in I}{\mathrm{convex}\text{
}\mathrm{hull}}\left(  \frac{1}{m}\underline{v}(h_{\alpha})\right)  \right)
=\sum_{\underline{w}}\sum_{c=0}^{n}(-1)^{c}[k\left(  \underline{w}\right)
:k]\label{luma0e}\\
&  \qquad\qquad\qquad\qquad\qquad\operatorname*{Vol}\left.  \mathrm{simplex}%
\left\langle \underline{w}(h_{\alpha_{0}(\underline{w})}),\ldots
\widehat{\underline{w}(h_{\alpha_{c}(\underline{w})})}\ldots,\underline
{w}(h_{\alpha_{n}(\underline{w})})\right\rangle \right.  \text{.}\nonumber
\end{align}
Here $\underline{w}$ runs over a finite set of top rank valuations
$\underline{w}:k\left(  X\right)  ^{\times}\rightarrow\mathbb{Z}%
_{\operatorname*{lex}}^{n}$ and there are well-defined values $\alpha
_{i}(\underline{w})\in I$ attached to each. We write $k\left(  \underline
{w}\right)  $ to denote the residue field of $\underline{w}$. Both sides of
the equation agree with the volume of the Newton--Okounkov body.
\end{maintheorem}

This is a shortened formulation of our Theorem \ref{thm_B}. The full version
of the theorem explains where the valuations $\underline{w}$ and the values
$\alpha_{i}(\underline{w})$ come from and how to determine them. It turns out
that they will be canonically determined once \textit{any} trivialization with
such properties is chosen. However, the full story is somewhat involved (it
involves the set $\mathcal{G}$ of Definition \ref{def_SetG} if you wish to
jump ahead).\medskip

Note that we have the global sections $h_{(-)}$ on both sides of the equation,
but only the chosen fixed valuation $\underline{v}$ on the left, while several
different valuations appear on the right. It is too complicated to describe
here how the $\underline{w}$ arise, but their choice is \textit{independent}
of the choice of $\underline{v}$. In particular, in a general situation,
$\underline{v}$ will not appear among the $\underline{w}$ on the right
side.\medskip

The statement of the main theorem is quite convoluted, so let us stress the
main point informally:

\begin{slogan}
If $D$ is very ample, there is a formula%
\[%
\begin{tabular}
[c]{lll}%
$\mathrm{Volume}$ $\mathrm{of}$ $\mathrm{Okounkov}$ $\mathrm{body}$ & $=$ &
$\mathbb{Z}$\textrm{-}$\mathrm{linear}\text{ }\mathrm{comb.}\text{
}\mathrm{of}\text{ }\mathrm{volumes}\text{ }\mathrm{of}\text{ }n$%
\textrm{-}$\mathrm{simplices}$\\
&  & \\
\multicolumn{1}{c}{$\mathrm{convex}$ $\mathrm{hull}$ $\mathrm{of}$
$\underline{v}(h_{(-)})$} & \multicolumn{1}{c}{} &
\multicolumn{1}{c}{$\mathrm{simplices}$ $\mathrm{with}$ $\mathrm{vertices}$
$\mathrm{at}$ $\underline{w}(h_{(-)})$}\\
\multicolumn{1}{c}{} & \multicolumn{1}{c}{} & \multicolumn{1}{c}{$\mathrm{for}%
$ $\mathrm{a}$ $\mathrm{finite}$ $\mathrm{number}$ $\mathrm{of}$
$\mathrm{valuations}$ $\underline{w}$,}%
\end{tabular}
\ \ \
\]
where the $h_{(-)}$ are global sections which restrict on opens of a finite
open cover to local trivializations of $D$ as a Cartier divisor.\medskip
\end{slogan}

We will explicitly evaluate both sides of our formula in a family of examples
on Hirzebruch surfaces. Just to get into the \textquotedblleft look
\&\ feel\textquotedblright\ how our formula may look in a concrete example
case:%
\begin{align}
&  \operatorname*{Vol}\left(  \mathrm{convex}\text{ }\mathrm{hull}\left(
\underline{v}(h_{0});\underline{v}(h_{2});\underline{v}(h_{4});0\right)
\right)  =\label{lvj}\\
&  \qquad\frac{1}{2}\det%
\begin{pmatrix}
\underline{w}_{1}(h_{4}) & \underline{w}_{1}(h_{2})\\
\underline{w}_{2}(h_{4}) & \underline{w}_{2}(h_{2})
\end{pmatrix}
-\frac{1}{2}\det%
\begin{pmatrix}
\underline{w}_{1}(h_{0}) & \underline{w}_{1}(h_{2})\\
\underline{w}_{2}(h_{0}) & \underline{w}_{2}(h_{2})
\end{pmatrix}
\nonumber\\
&  \qquad+\frac{1}{2}\det%
\begin{pmatrix}
\underline{w}_{1}(h_{0}) & \underline{w}_{1}(h_{4})\\
\underline{w}_{2}(h_{0}) & \underline{w}_{2}(h_{4})
\end{pmatrix}
-\frac{1}{2}\det%
\begin{pmatrix}
\underline{w}_{1}^{\prime}(h_{0}) & \underline{w}_{1}^{\prime}(h_{4})\\
\underline{w}_{2}^{\prime}(h_{0}) & \underline{w}_{2}^{\prime}(h_{4})
\end{pmatrix}
\text{.}\nonumber
\end{align}
We refer to \S \ref{sect_Example} for notation and details. Three top rank
valuations $\underline{v}$, $\underline{w}$ and $\underline{w}^{\prime}$
appear in this formula.\medskip

We have chosen a particularly provocative formulation in Equation
\ref{luma0e}. In principle we are only comparing volumes, so we could also
have spelled out the volume of the simplices on the right-hand side. However,
since the vertices entering the convex geometry on the left and right side
look so similar, I would dream that this equality might just be a
\textquotedblleft shadow\textquotedblright\ of a stronger identity. For
example, writing%
\[%
\begin{array}
[c]{cc}%
\triangle_{\underline{w}}:= &
\raisebox{-0.4295in}{\includegraphics[
height=0.8832in,
width=1.0288in
]%
{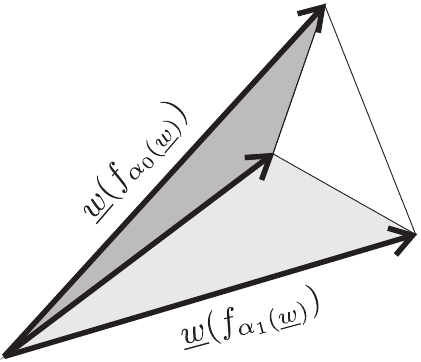}%
}%
\end{array}
\]
for the $n$-simplices on the right-hand side of the formula, one could imagine
that there might be a systematic way to assemble these simplices to a polytope
in their own right. So far all my attempts to properly understand this have
been futile, so this remains a wild dream.\medskip

The right hand side of the formula in the theorem is also valid without the
assumption that $\Gamma_{\underline{v}}(D)$ is finitely generated. However,
then the Newton--Okounkov body does not need to be a polytope and we lose the
left side of the equation. We still get the correct volume though (this will
be Theorem \ref{thm_MainFormula}).\medskip

Our method is based on not very well-known approaches to intersection theory,
notably mixing ideas of Bloch \cite{MR0342514} and Parshin \cite{MR697316},
\cite{ParshinAdeleTheory}. These approaches have a special strength: they do
not have any trouble if the set-theoretic dimension of the intersection is
wrong, e.g. for negative self-intersections, and at the same time they remain
local (unlike Riemann--Roch style ideas, like Snapper, Kleiman's via $\chi$)
and algebraic, unlike for example singular cohomology. Concretely, we use the
framework of Bloch, which uses \v{C}ech cohomology instead of the less
well-known ad\`{e}les of schemes in Parshin's approach. Nonetheless, the ideas
how this relates to top rank valuations are imported from Parshin's method. We
wrote this text in a self-contained way, so no previous familiarity with these
methods is assumed.\medskip

In this text a `variety' is a scheme which is finite type and separated over
some field. Natural numbers begin with zero, $\mathbb{N}:=\mathbb{Z}_{\geq0}$.

\section{Valuation vectors and\ Okounkov body}

Let $k$ be a field. Suppose $X/k$ is an $n$-dimensional irreducible smooth
projective variety. Let $v$ be a top rank valuation on $X$ (i.e. of rank equal
to the dimension $n$). It is necessarily discrete with (the lexicographically
ordered) value group $\mathbb{Z}_{\operatorname*{lex}}^{n}$, see \cite[Ch. VI,
p. 90, Corollary]{MR0389876}; we write $\underline{v}:k\left(  X\right)
^{\times}\rightarrow\mathbb{Z}_{\operatorname*{lex}}^{n}$. We also call
$\underline{v}(f)=(\underline{v}_{1}(f),\ldots,\underline{v}_{n}%
(f))\in\mathbb{R}_{\operatorname*{lex}}^{n}$ a \emph{valuation vector} and the
$\underline{v}_{i}(f)\in\mathbb{Z}$ its \emph{components}. We shall always
tacitly assume that our valuations are trivial on the base field $k$. One way
to produce such a valuation is to pick an admissible flag:

\begin{definition}
\label{def_Flag}A\ \emph{flag} in $X$ is the datum%
\[
Y_{\bullet}:X=Y_{0}\supset Y_{1}\supset Y_{2}\supset\cdots\supset
Y_{0}\text{,}%
\]
where $Y_{i}$ is an integral closed subscheme of codimension $i$. In
particular, $Y_{0}$ is a single closed point. Following \cite{MR2571958}, we
call the flag \emph{admissible} if each $Y_{i}$ is regular at $Y_{0}$.
\end{definition}

While \cite{MR2571958} focusses on admissible flags, the entire theory of
Newton--Okounkov bodies can be formulated with respect to arbitrary top rank
valuations on the rational function field of a variety. This is very nicely
explained in \cite{MR3694973}.

The \emph{graded semigroup} attached to a divisor $D$ is%
\begin{equation}
\Gamma_{\underline{v}}(D):=\left\{  \left.  (\underline{v}(s),m)\right\vert
s\in H^{0}(X,\mathcal{O}_{X}(mD))\setminus\{0\}\text{ for any }m\geq0\right\}
\subseteq\mathbb{Z}^{n}\times\mathbb{N}\text{.} \label{lima5a}%
\end{equation}

\begin{definition}
\label{def_Okounkov}The \emph{(Newton--)Okounkov body} $\Delta_{\underline{v}%
}(D)\subset\mathbb{R}^{n}$ is given by%
\[
\Delta_{\underline{v}}(D):=\overline{\left.  \mathrm{convex}\text{
}\mathrm{hull}\right.  \left(  \Gamma_{\underline{v}}(D)\right)  }%
\cap(\mathbb{R}^{n}\times\{1\})\text{,}%
\]
i.e. we take the closed convex hull of all the points in $\Gamma
_{\underline{v}}(D)$ inside $\mathbb{R}^{n}\times\mathbb{R}$, and then
intersect with the coordinate hyperplane $m=1$. Alternatively,%
\begin{equation}
\Delta_{\underline{v}}(D):=\overline{\left.  \mathrm{convex}\text{
}\mathrm{hull}\right.  \left\{  \left.  \frac{1}{m}\underline{v}(s)\right\vert
s\in H^{0}(X,\mathcal{O}_{X}(mD))\setminus\{0\}\text{ for }m\geq1\right\}  }
\label{lima5}%
\end{equation}
inside $\mathbb{R}^{n}$.
\end{definition}

\begin{history}
The original definition is due to Okounkov \cite{MR1400312}, \cite{MR1995384},
but with a focus on a different kind of question. Both \cite{MR2571958} as
well as \cite{MR2950767} have brought the concept to the midst of birational geometry.
\end{history}

For $m=0$ observe that $\Gamma_{\underline{v}}(D)$ contains the origin
$(0,\ldots,0)$, so as soon as it contains the point $(\underline{v}(s),m)$,
the connecting segment to the origin crosses the $m=1$ hyperplane at
$(\frac{1}{m}\underline{v}(s),1)$. This essentially shows the equivalence of
both definitions.

\begin{example}
If the graded semigroup $\Gamma_{\underline{v}}(D)$ is finitely generated,
$\Delta_{\underline{v}}(D)$ is a rational polytope. Examples with
$\Gamma_{\underline{v}}(D)$ finitely generated are rather rare, but see
\cite[Proposition 14]{MR3207370} for a construction of examples.
\end{example}

We begin with an elementary observation.

\begin{proposition}
\label{prop_TrivPolytope}Let $X/k$ be an irreducible smooth projective
variety, $\underline{v}$ a top rank valuation and $D$ an ample divisor. Pick
$m\geq1$ such that $mD$ is very ample. Then there exists a local
trivialization $(U_{\alpha},h_{\alpha})_{\alpha\in I}$ of $\mathcal{O}%
_{X}(mD)$ as a Cartier divisor with $I$ a finite index set and

\begin{enumerate}
\item each $h_{\alpha}$ on $U_{\alpha}$ is the restriction of a global
section; for simplicity we write $h_{\alpha}\in H^{0}(X,\mathcal{O}_{X}(mD))$;

\item $(h_{\alpha})_{\alpha}$ with $h_{\alpha}\in H^{0}(U_{\alpha}%
,\mathcal{K}_{X}^{\times}/\mathcal{O}_{X}^{\times})$ is a Cartier divisor
representative of $mD$;

\item and moreover the \emph{trivialization polytope}%
\begin{equation}
\left.  \underset{\alpha\in I}{\mathrm{convex}\text{ }\mathrm{hull}}\right.
\left(  \frac{1}{m}\underline{v}(h_{\alpha})\right)  \subseteq\Delta
_{\underline{v}}(D)\text{,} \label{luma2}%
\end{equation}
is contained inside the Newton--Okounkov body.
\end{enumerate}
\end{proposition}

\begin{proof}
Direct. Once $mD$ is very ample, it is globally generated.
\end{proof}

\begin{example}
In general the polytope of Equation \ref{luma2} will be strictly smaller than
$\Delta_{\underline{v}}(D)$. Necessarily so, because the construction always
yields a polytope, yet there are examples of Newton--Okounkov bodies which are
non-polyhedral, \cite[\S 6.3]{MR2571958}. Such examples can only exist in
dimension $\geq3$, and further such examples can be found in \cite[\S 3]%
{MR2889138}.
\end{example}

\begin{example}
\label{example_ToricF}Quite on the contrary, for an ample divisor on a toric
variety and $\underline{v}$ coming from a $T$-invariant flag, we can achieve
equality in Equation \ref{luma2}. Indeed, the Newton--Okounkov body always
agrees with the divisor polytope, so it is necessarily polyhedral, as follows
from \cite[Proposition 6.1, (i)]{MR2571958}. Let $X/k$ be a smooth proper
toric variety. Let $D$ be any $T$-invariant divisor, i.e.%
\[
D=\sum_{\sigma}n_{\sigma}D_{\sigma}\text{,}%
\]
where $\sigma\in\Sigma(1)$ runs through the rays of the polyhedral fan
$\Sigma\subset N_{\mathbb{R}}$ such that $X=X(\Sigma)$. Following the notation
of Fulton \cite{MR1234037}, let $M:=\operatorname*{Hom}(N,\mathbb{Z})$. Then
if $u_{\alpha}$ runs through the maximal cones of $\Sigma$, on each affine
open $U_{\alpha}:=\operatorname*{Spec}k[u_{\alpha}^{\vee}\cap M]$ the line
bundle $\mathcal{O}(D)$ is trivialized $\left.  \mathcal{O}(D)\right\vert
_{U_{\alpha}}=\left.  \frac{1}{h_{\alpha}}\mathcal{K}_{X}\right\vert
_{U_{\alpha}}$ by the unique monomial $h_{\alpha}\in M$ determined by%
\begin{equation}
\left\langle h_{\alpha},\sigma\right\rangle =-n_{\sigma}\text{,} \label{luma4}%
\end{equation}
where $\sigma$ runs through all rays of the cone $u_{\alpha}$. As $X$ is
smooth, and thus $\sigma$ a smooth cone, this amounts to a system of equations
with a unique solution. We have%
\[
H^{0}(X,\mathcal{O}_{X}(D))=k\text{-}\operatorname*{span}\left\langle h\in
M\mid\left\langle h,\sigma\right\rangle \geq-n_{\sigma}\text{ for all }%
\sigma\in\Sigma(1)\right\rangle \,\text{,}%
\]
the usual description of the sections via the divisor polytope $P_{D}$,
\cite[p. 66]{MR1234037}. As is discussed in the proof of \cite[p. 68,
Proposition]{MR1234037}, $\mathcal{O}(D)$ is generated by global sections if
and only if for each $\alpha$, $\left\langle h,\sigma\right\rangle
\geq-n_{\sigma}$ holds for all rays $\sigma$ (also those not being faces of
the cone $u_{\alpha}$). If $D$ is ample, this is satisfied, \cite[p. 70,
Proposition]{MR1234037}. This puts us in the situation described in
Proposition \ref{prop_TrivPolytope}, but with equality of polytopes in
Equation \ref{luma2}. We depict below on the left the polyhedral fan in
$N_{\mathbb{R}}$, the rays being numbered $\sigma_{1},\sigma_{3},\ldots$ with
odd indices, and the dotted lines represent each of the Equations \ref{luma4}.
The shaded area is the divisor polytope, and equivalently, up to identifying
spaces as in \cite[Proposition 6.1, (i)]{MR2571958}, the Newton--Okounkov
body.
\[%
{\includegraphics[
height=1.2489in,
width=3.7576in
]%
{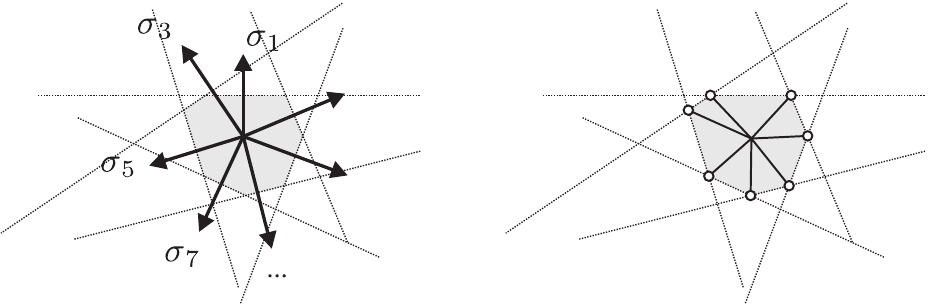}%
}%
\]
The bullet points above on the right correspond to the functions $h_{\alpha}$.
Their convex hull agrees with the divisor polytope. Having the subject of this
paper in mind, these bullet points amount to what we want to have on the left
side of Equation \ref{luma0e} in the introduction.
\end{example}

The idea of the previous example also works for general varieties if the
graded semigroup is finitely generated.

\begin{lemma}
\label{Lemma_Triv}We make the same assumptions as in Proposition
\ref{prop_TrivPolytope}. If the graded semigroup $\Gamma_{\underline{v}}(D)$
is finitely generated, then the claim of Proposition \ref{prop_TrivPolytope}
is true with the additional property that the trivialization polytope is the
entire Newton--Okounkov body.
\end{lemma}

\begin{proof}
(Step 1) Let $(\underline{v}(s_{i}),m_{i})$ be a finite set of semigroup
generators for $\Gamma_{\underline{v}}(D)$. We claim that we may simplify
Equation \ref{lima5} to%
\begin{equation}
\Delta_{\underline{v}}(D)=\left.  \mathrm{convex}\text{ }\mathrm{hull}\right.
\left\{  \frac{1}{m_{i}}\underline{v}(s_{i})\right\}  \text{.} \label{lima5b}%
\end{equation}
This holds since any $(\underline{v}(s),m)$ in Equation \ref{lima5a} can be
written as%
\[
(\underline{v}(s),m)=\sum_{i}n_{i}(\underline{v}(s_{i}),m_{i})\qquad
\text{for}\qquad n_{i}\in\mathbb{Z}_{\geq1}\text{.}%
\]
The line segment from $(0,\ldots,0)$ to $(\underline{v}(s),m)$ intersects the
$m=1$ hyperplane at%
\[
\frac{1}{m}(\underline{v}(s),m)=\frac{1}{\sum_{j}n_{j}m_{j}}\left(  \sum
_{i}n_{i}\underline{v}(s_{i}),\sum_{i}n_{i}m_{i}\right)  =\left(  \sum
_{i}\left(  \frac{n_{i}\cdot m_{i}}{\sum n_{j}m_{j}}\right)  \frac{1}{m_{i}%
}\underline{v}(s_{i}),1\right)
\]
and we have $\frac{n_{i}\cdot m_{i}}{\sum n_{j}m_{j}}\in\lbrack0,1]$ and the
sum of all these coefficients is $1$, so this is a convex combination of the
vectors in Equation \ref{lima5b}. Being the convex hull of finitely many
points, we do not need to take the topological closure in Equation
\ref{lima5b}. We have $\operatorname*{div}(s_{i})\geq-m_{i}D$ and since the
divisor is supported in a closed subscheme of codimension $\geq1$, we may pick
a dense open $U$ such that we have equality $\left.  \operatorname*{div}%
(s_{i})\right\vert _{U}=\left.  -mD\right\vert _{U}$. We may add this open to
any given finite trivialization; repeat this for all $i$.\newline(Step
2)\ Since $\frac{1}{m_{i}}\underline{v}(s_{i})=\frac{1}{m_{i}N}\underline
{v}(s_{i}^{N})$ for any $N\geq1$, we may pick a common multiple of the
denominators, and then Equation \ref{lima5b} can be rewritten as
$\Delta_{\underline{v}}(D)=\left.  \mathrm{convex}\text{ }\mathrm{hull}%
\right.  \left\{  \frac{1}{N}\underline{v}(s_{i}^{\prime})\right\}  $ for
suitable $s_{i}^{\prime}:=s_{i}^{N/m_{i}}\in H^{0}(X,\mathcal{O}_{X}(N\cdot
D))$ and $N$ independent of $i$. We have arrived at the claim of Proposition
\ref{prop_TrivPolytope}, but now the trivialization polytope agrees with
$\Delta_{\underline{v}}(D)$ instead of merely being contained in it.
\end{proof}

\begin{example}
Under similar assumptions as in Lemma \ref{Lemma_Triv}, Anderson constructs a
family $\tilde{X}\rightarrow\mathbb{A}^{1}$ such that the central fiber
$\tilde{X}_{0}$ (up to normalization) agrees with the toric variety of the
polytope $\Delta_{Y_{\bullet}}(D)$, and all other fibers $\tilde{X}_{t}$,
$t\neq0$, are isomorphic to $X$, \cite{MR3063911}. Along this degeneration,
Lemma \ref{Lemma_Triv} transforms into the situation discussed in Example
\ref{example_ToricF}.
\end{example}

\section{The intersection form}

\subsection{Intersection form via topology\label{sect_IntersectionForm}}

Let us first recall some basics around intersection numbers. In his book
\cite{MR2095471}, Lazarsfeld uses the following elementary approach: Suppose
$X/\mathbb{C}$ is an $n$-dimensional integral proper variety. Then we may
attach to any line bundle $L$ its first Chern class $c_{1}(L)\in
H^{2}(X,\mathbb{Z})$, which can in turn be defined as the image of the
connecting map of the exponential sequence of sheaves%
\begin{equation}
0\longrightarrow\mathbb{Z}\overset{\cdot2\pi i}{\longrightarrow}%
\mathcal{O}_{X}\overset{\exp}{\longrightarrow}\mathcal{O}_{X}^{\times
}\longrightarrow0\text{,} \label{lta1}%
\end{equation}
namely%
\[
H^{1}(X,\mathcal{O}_{X}^{\times})\longrightarrow H^{2}(X,\mathbb{Z})\text{.}%
\]
Now, just using the product in the singular cohomology ring, one obtains the
intersection form%
\[
\int_{X}L_{1}\cdots L_{n}:=\int_{X}c_{1}(L_{1})\smile\cdots\smile c_{1}%
(L_{n})\in\mathbb{Z}\text{,}%
\]
where the integral on the right denotes the evaluation against the fundamental
class $[X]\in H_{2n}(X,\mathbb{Z)}$. In other words, the intersection form
arises as the composition of%
\begin{align*}
&  H^{1}(X,\mathcal{O}_{X}^{\times})\otimes\cdots\otimes H^{1}(X,\mathcal{O}%
_{X}^{\times})\\
&  \qquad\qquad\longrightarrow H^{2}(X,\mathbb{Z})\otimes\cdots\otimes
H^{2}(X,\mathbb{Z})\overset{\smile}{\longrightarrow}H^{2n}(X,\mathbb{Z}%
)\overset{\cdot\lbrack X]}{\longrightarrow}\mathbb{Z}\text{.}%
\end{align*}
In this text, we shall use a different approach based on Milnor $K$-theory. It
is based on writing Chow groups as sheaf cohomology groups:

\begin{theorem}
[Bloch--Quillen, Grayson]\label{thm_BQ}Suppose $X/k$ is a smooth variety. Then
there is a canonical ring isomorphism%
\[
\operatorname*{CH}\nolimits^{\ast}(X)\cong H_{\operatorname*{Zar}}^{\ast
}(X,\mathcal{K}_{\ast}^{\operatorname*{M}})\text{,}%
\]
where $\mathcal{K}_{\ast}^{\operatorname*{M}}$ is the Milnor $K$-theory sheaf
(whose definition we recall below). On the right-hand side we use sheaf
cohomology with respect to the Zariski topology on $X$.
\end{theorem}

\begin{history}
This method is not very well-known, so let us provide some background: In
degree one it boils down to $\operatorname*{CH}\nolimits^{1}(X)\cong
H_{\operatorname*{Zar}}^{\ast}(X,\mathcal{O}_{X}^{\times})$, i.e. the
classical isomorphism between Weil divisor classes and Cartier divisor
classes. The idea to generalize this to higher codimension is due to Bloch,
who discovered that codimension two algebraic cycles are related to $K_{2}%
$-groups; he even speaks of \textquotedblleft codimension two Cartier
divisors\textquotedblright\ in \cite{MR0342514}. Quillen proved the general
version. These papers used full algebraic $K$-theory, but Kato later
discovered that Milnor $K$-theory also works. Grayson proved that the
intersection product also matches the sheaf cup product structure on the right
side in the context of algebraic $K$-theory \cite{MR0491685}. In this paper we
use a more modern formulation due to Rost \cite{MR1418952} using Milnor
$K$-theory which has nice signs. See also Gillet \cite{MR2181825}.
\end{history}

Let us carefully explain how to use the above theorem for the computation of
intersection numbers. But first we should recall the definition of the sheaves
$\mathcal{K}_{\ast}^{\operatorname*{M}}$, which we will do in the next section.

\subsection{Milnor $K$-groups}

Let $F$ be any field. Then $F^{\times}$ is an abelian group, and thus a
$\mathbb{Z}$-module. Write%
\[
TF^{\times}:=\bigoplus_{p\geq0}\underset{p\text{ factors}}{\underbrace
{F^{\times}\otimes_{\mathbb{Z}}\cdots\otimes_{\mathbb{Z}}F^{\times}}}%
\]
for the free tensor algebra of $F^{\times}$, as a $\mathbb{Z}$-module. The
unusal r\^{o}le of using powers to define the $\mathbb{Z}$-module structure
can be confusing at first. For example, for any integer $n\in\mathbb{Z}$, it
means that%
\begin{equation}
n\cdot(f_{1}\otimes f_{2})=f_{1}^{n}\otimes f_{2}=f_{1}\otimes f_{2}^{n}%
\qquad\qquad\qquad f_{1}\otimes\frac{1}{f_{2}}=\frac{1}{f_{1}}\otimes
f_{2}\text{.} \label{lt1}%
\end{equation}
Now define the \emph{Milnor }$K$\emph{-theory ring} as%
\begin{equation}
K_{\ast}^{\operatorname*{M}}(F):=\frac{TF^{\times}}{\left\langle
x\otimes(1-x)\mid\text{for all }x\in F\setminus\{0,1\}\right\rangle }\text{,}
\label{lt1_z}%
\end{equation}
i.e. we quotient out the two-sided ideal generated by all pure tensors
$x\otimes(1-x)$ for $x\in F\setminus\{0,1\}$. This is called the
\emph{Steinberg relation}. It is customary to write%
\[
\{f_{1},\ldots,f_{p}\}:=f_{1}\otimes\cdots\otimes f_{p}%
\]
for the elements in this quotient ring. We obtain that $K_{\ast}%
^{\operatorname*{M}}(F)$ is a $\mathbb{Z}$-graded ring. In degrees $0$ and $1$
it is easy to understand%
\[
K_{\ast}^{\operatorname*{M}}(F)=\mathbb{Z}\oplus F^{\times}\oplus
K_{2}^{\operatorname*{M}}(F)\oplus K_{3}^{\operatorname*{M}}(F)\oplus
\ldots\text{.}%
\]
In general $K_{p}^{\operatorname*{M}}(F)$ for $p\geq2$ is rather uncomputable
except for very special cases. However, we will never really need to know what
these groups are, so this is harmless.

\begin{lemma}
[{\cite[Lemma 1.1 and 1.2]{MR0260844}}]\label{lemma_GradedCommutative}%
$K_{\ast}^{\operatorname*{M}}(F)$ is graded-commutative, and moreover
$\{x,x\}=\{x,-1\}$ holds for any $x\in F^{\times}$.
\end{lemma}

\begin{lemma}
\label{lemma_StdShapeForValuation}Let $F$ be a field and $v$ a discrete rank
one valuation. Then every element in $K_{p}^{\operatorname*{M}}(F)$ is a
finite $\mathbb{Z}$-linear combination of elements of the shapes%
\[
\{\pi,u_{2},\ldots,u_{p}\}\qquad\text{and}\qquad\{u_{1},u_{2},\ldots
,u_{p}\}\text{,}%
\]
where $\pi$ is a uniformizer, i.e. $v(\pi)=1$, and $u_{1},\ldots,u_{p}%
\in\mathcal{O}_{v}^{\times}$ have valuation zero.
\end{lemma}

\begin{proof}
It suffices to prove this for a pure tensor $\{f_{1},\ldots,f_{p}\}$. Every
element $f\in F^{\times}$ can be written as $u\pi^{k}$ for $u\in
\mathcal{O}_{v}^{\times}$, $k\in\mathbb{Z}$. Do this in every slot. Then
multi-linearity and the relation $\{\pi,\pi\}=\{\pi,-1\}$ (Lemma
\ref{lemma_GradedCommutative}) yield the claim.
\end{proof}

We need a few more constructions.

\textbf{C1} \textit{(Restriction)} If $L/F$ is an arbitrary field extension,
there is a natural graded ring homomorphism%
\[
\operatorname*{res}\nolimits_{L/F}:K_{\ast}^{\operatorname*{M}}%
(F)\longrightarrow K_{\ast}^{\operatorname*{M}}(L)
\]
simply by using the inclusion $F^{\times}\hookrightarrow L^{\times}$ slot by slot.

\textbf{C2 }\textit{(Norms)\footnote{This is called `corestriction' or `norm'
in the literature, depending on the taste of authors.}} If $L/F$ is a
\textit{finite} field extension, there is a natural map%
\begin{equation}
N_{L/K}:K_{p}^{\operatorname*{M}}(L)\longrightarrow K_{p}^{\operatorname*{M}%
}(F) \label{lt2a}%
\end{equation}
in each degree $p$. It does not preserve the ring structure. In degree $0$ it
is $\mathbb{Z}\overset{\cdot\lbrack L:F]}{\longrightarrow}\mathbb{Z}$, and in
degree $1$ it is the usual norm map $L^{\times}\longrightarrow F^{\times}$.
The construction of this map in higher degrees is complicated. We refer to
\cite[\S 7.3]{MR2266528}. The norm is functorial in towers, i.e. if%
\[
F\subseteq L\subseteq L^{\prime}%
\]
are two finite field extensions,%
\begin{equation}
N_{L^{\prime}/F}=N_{L/F}\circ N_{L^{\prime}/L}\text{.} \label{l_NormInTowers}%
\end{equation}

\textbf{C3 }\textit{(Boundary map) }If $v$ is a discrete rank one valuation on
$F$, let $\mathcal{O}_{v}\subset F$ denote its valuation ring with maximal
ideal $\mathfrak{m}_{v}$. Write $\kappa(v):=\mathcal{O}_{v}/\mathfrak{m}_{v}$
for its residue field. There is a natural map%
\begin{equation}
\partial_{v}:K_{p}^{\operatorname*{M}}(F)\longrightarrow K_{p-1}%
^{\operatorname*{M}}(\kappa(v))\text{.} \label{lt3a}%
\end{equation}
In degree $1$ it is $F^{\times}\longrightarrow\mathbb{Z}$, $f\mapsto v(f)$,
the valuation itself. For $p\geq2$, it can fully be described as follows: If
$\pi$ is a uniformizer for $v$, i.e. $v(\pi)=1$, and $u_{1},\ldots,u_{p}%
\in\mathcal{O}_{v}^{\times}$ then define%
\begin{align}
\partial_{v}\{\pi,u_{2},\ldots,u_{p}\}  &  :=\{\overline{u_{2}},\ldots
,\overline{u_{p}}\}\label{lt3}\\
\partial_{v}\{u_{1},u_{2},\ldots,u_{p}\}  &  :=0\text{.} \label{lt3b}%
\end{align}
and by Lemma \ref{lemma_StdShapeForValuation} this suffices to uniquely pin
down the map. We neglect proving here that this amounts to a well-defined map.
The map is independent of the choice of the uniformizer $\pi$.

There are various relations between the maps of \textbf{C1}-\textbf{C3}. We
shall only need very few. Most importantly, we will use the compatibility
between boundary maps and norms.

\begin{proposition}
[{\cite[Ch. 7, Corollary 7.4.3]{MR2266528}}]\label{prop_NormCompat}Suppose $L$
is a discrete valuation field with valuation $v$ and residue field $\kappa
(v)$. Write $\mathcal{O}_{v}$ for its valuation ring. Let $L^{\prime}/L$ be a
finite field extension. Write $\mathcal{O}_{v}^{\prime}$ for the integral
closure of the valuation ring $\mathcal{O}_{v}$ inside $L^{\prime}$. Suppose
$\mathcal{O}_{v}^{\prime}$ is a finite $\mathcal{O}_{v}$-module (this will
hold whenever we have finiteness of integral closures, e.g. if $\mathcal{O}%
_{v}$ is a valuation ring coming from a variety). Then $\mathcal{O}%
_{v}^{\prime}$ is a semi-local ring and if $w$ runs through the discrete
valuations extending $v$, the diagram%
\begin{equation}
\xymatrix{ K^{\operatorname{M}}_{\ast}(L^{\prime}) \ar[r]^-{\oplus{\partial}_w} \ar[d]_{N_{L^{\prime} /L}} & \bigoplus_{w} K^{\operatorname{M}}_{\ast- 1}(\kappa(w)) \ar[d]^{\Sigma N_{\kappa(w) / \kappa (v)}} \\ K^{\operatorname{M}}_{\ast}(L) \ar[r]_-{{\partial}_v} & K^{\operatorname{M}}_{\ast- 1}(\kappa(v)) }
\label{lNormSquare}%
\end{equation}
commutes.
\end{proposition}

\subsection{The boundary map between points\label{sect_BdyMapsBetweenPoints}}

We need to recall a refinement of the boundary map $\partial_{v}$ for schemes.
The main point is that we wish to attach a boundary map to any codimension one
subscheme, but outside the normal locus, a codimension one scheme does not
uniquely pin down a discrete rank one valuation. The construction is a little
involved. The reader may prefer to skip it initially, keeping in mind that
`generically' Remark \ref{rmk_SimplerSituation} can be used instead.

Let $Y/k$ be an integral variety. Suppose $Z\hookrightarrow Y$ is an integral
closed subscheme of codimension one. The underlying closed subset of $Z$ is
irreducible and thus has a unique generic point, call it $z$. Being of
codimension one, the local ring $\mathcal{O}_{Y,z}$ is a one-dimensional local
domain. It is contained in the rational function field $k\left(  Y\right)  $
(the local ring at the generic point of $Y$ itself). Its residue field
$\kappa(z)$ agrees with the rational function field $k\left(  Z\right)  $ of
$Z$.

Let $\mathcal{O}_{Y,z}^{\prime}$ denote the integral closure of $\mathcal{O}%
_{Y,z}$ inside $k\left(  Y\right)  $, i.e.%
\begin{equation}
\operatorname*{Spec}\mathcal{O}_{Y,z}^{\prime}\longrightarrow
\operatorname*{Spec}\mathcal{O}_{Y,z} \label{lt8}%
\end{equation}
is the normalization morphism. This is a finite morphism of schemes since $Y$
was assumed to be of finite type over a field (finiteness of integral
closure). Thus, the ring $\mathcal{O}_{Y,z}^{\prime}$ is semi-local. Let
$\mathfrak{m}_{1},\ldots,\mathfrak{m}_{r}$ denote its maximal ideals. Then
each localization $(\mathcal{O}_{Y,z}^{\prime})_{\mathfrak{m}_{i}}\subset
k\left(  Y\right)  $ is a one-dimensional normal domain and thus a discrete
valuation ring. Let $v_{i}$ be the discrete rank one valuation of $k\left(
Y\right)  $ coming from $(\mathcal{O}_{Y,z}^{\prime})_{\mathfrak{m}_{i}}$. The
residue field of the valuation $v_{i}$ agrees with the residue field of
$(\mathcal{O}_{Y,z}^{\prime})_{\mathfrak{m}_{i}}$, so we call it $\kappa
(v_{i})$. Since each $\mathfrak{m}_{i}$ under the finite morphism of Equation
\ref{lt8} maps to $\mathfrak{m}_{z}$, we get attached residue field extensions
$\kappa(v_{i})/\kappa(z)$ and these are themselves finite.

\begin{definition}
\label{def_BdyBetweenPoints}Let $Y/k$ be an integral variety. Suppose
$Z\hookrightarrow Y$ is an integral closed subscheme of codimension one.
Define%
\[
\partial_{Z}^{Y}:K_{p}^{\operatorname*{M}}(k\left(  Y\right)  )\overset
{\oplus\partial_{v_{i}}}{\longrightarrow}\bigoplus_{i=1}^{r}K_{p-1}%
^{\operatorname*{M}}(\kappa(v_{i}))\overset{\oplus N_{\kappa(v_{i})/k\left(
Z\right)  }}{\longrightarrow}K_{p-1}^{\operatorname*{M}}(k\left(  Z\right)
)\text{,}%
\]
where the first arrow stems from the boundary maps $\partial_{v_{i}}$ of
Equation \ref{lt3a}, and each $N_{(-)/(-)}$ is a norm map as in Equation
\ref{lt2a}.
\end{definition}

\begin{remark}
\label{rmk_SimplerSituation}If $\mathcal{O}_{Y,z}$ is normal, it is itself a
discrete valuation ring. Write $v$ for its valuation. In this special case the
above definition simplifies to $\partial_{Z}^{Y}=\partial_{v}$, i.e. we are
back in the situation of Equation \ref{lt3a}.
\end{remark}

\subsection{The Milnor $K$-sheaf}

Let $X/k$ be a variety.

\begin{definition}
For any Zariski open $U\subseteq X$ define%
\begin{equation}
\mathcal{K}_{p}^{\operatorname*{M}}(U):=\ker\left(  \bigoplus_{Y_{0}}%
K_{p}^{\operatorname*{M}}(k\left(  Y_{0}\right)  )\overset{\partial_{Y_{1}%
}^{Y_{0}}}{\longrightarrow}\bigoplus_{Y_{1}}K_{p-1}^{\operatorname*{M}%
}(k\left(  Y_{1}\right)  )\right)  \text{,} \label{lta4}%
\end{equation}
where $Y_{i}\subseteq U$ runs through all integral closed subschemes of
codimension $i$ (for $i=0,1$). The map $\partial_{Y_{1}}^{Y_{0}}$ was set up
in Definition \ref{def_BdyBetweenPoints}.
\end{definition}

Note that if $X$ is integral, this simplifies to%
\[
\ker\left(  K_{p}^{\operatorname*{M}}(k\left(  X\right)  )\overset
{\partial_{Y_{1}}^{X}}{\longrightarrow}\bigoplus_{Y_{1}}K_{p-1}%
^{\operatorname*{M}}(k\left(  Y_{1}\right)  )\right)  \subseteq K_{p}%
^{\operatorname*{M}}(k\left(  X\right)  )\text{,}%
\]
so if we go to a smaller open $U^{\prime}\subseteq U$, we at worst remove
conditions when forming the kernel, so it is clear that $\mathcal{K}%
_{p}^{\operatorname*{M}}(-)$ defines a presheaf. Argueing individually for
each irreducible component, this works in general without requiring $X$ to be
integral. One can further check that $\mathcal{K}_{p}^{\operatorname*{M}}$
defines a sheaf in the Zariski topology. There is a product structure%
\begin{equation}
\smile\colon\mathcal{K}_{p}^{\operatorname*{M}}\otimes\mathcal{K}%
_{q}^{\operatorname*{M}}\longrightarrow\mathcal{K}_{p+q}^{\operatorname*{M}}
\label{lta4b}%
\end{equation}
making $\mathcal{K}_{\ast}^{\operatorname*{M}}:=\bigoplus_{p\geq0}%
\mathcal{K}_{p}^{\operatorname*{M}}$ a Zariski sheaf of graded commutative
algebras. The product is defined simply by concatenation (and thus compatible
to the one induced from the tensor algebra in the definition for fields,
Equation \ref{lt1_z}),%
\[
\{x_{1},\ldots,x_{p}\}\smile\{y_{1},\ldots,y_{q}\}:=\{x_{1},\ldots,x_{p}%
,y_{1},\ldots,y_{q}\}\text{,}%
\]
but one must check that this still lies in the kernel as in Equation
\ref{lta4}. This works as follows.

If $v$ is a discrete rank one valuation on a field $F$ and $\pi\in
\mathcal{O}_{v}$ a chosen uniformizer, one can define the \emph{specialization
map}%
\begin{equation}
s_{v}^{\pi}:K_{p}^{\operatorname*{M}}(F)\longrightarrow K_{p}%
^{\operatorname*{M}}(\kappa(v))\text{,}\qquad\qquad\{x_{1},\ldots
,x_{p}\}\longmapsto\partial_{v}\{-\pi,x_{1},\ldots,x_{p}\}\text{.}
\label{ltya2}%
\end{equation}
Unlike $\partial_{v}$, these maps depend on the uniformizer $\pi$. Based on
this, there is an extended \textquotedblleft non-linear\textquotedblright%
\ Leibniz formula for the maps $\partial_{v}$: For $x\in K_{p}%
^{\operatorname*{M}}(F)$ and $y\in K_{q}^{\operatorname*{M}}(F)$, we have%
\[
\partial_{v}(x\cdot y)=\partial_{v}(x)\cdot s_{v}^{\pi}(y)+(-1)^{p}s_{v}^{\pi
}(x)\cdot\partial_{v}(y)+\{-1\}\cdot\partial_{v}(x)\cdot\partial
_{v}(y)\text{,}%
\]
valid irrespective of the choice of the uniformizer. Thus, if $\partial
_{v}(x)=0$ and $\partial_{v}(y)=0$, it follows that $\partial_{v}(x\cdot
y)=0$, giving Equation \ref{lta4b}.

\begin{lemma}
\label{lemma_IsoOfSheaves}Suppose $X$ is normal. There is an isomorphism of sheaves

\begin{enumerate}
\item $\mathcal{K}_{0}^{\operatorname*{M}}\cong\mathbb{Z}$,

\item $\mathcal{K}_{1}^{\operatorname*{M}}\cong\mathcal{O}_{X}^{\times}$.
\end{enumerate}
\end{lemma}

\begin{proof}
(1) follows since $K_{0}^{\operatorname*{M}}(F)=\mathbb{Z}$ and $K_{-1}%
^{\operatorname*{M}}(F)=0$ for any field $F$. The direct sum over $Y_{0}$ in
Equation \ref{lta4} is exactly the same which would arise from Zariski
sheafification of the constant presheaf $\mathbb{Z}$. (2) We find%
\[
\mathcal{K}_{1}^{\operatorname*{M}}(U):=\ker\left(  \bigoplus_{Y_{0}}k\left(
Y_{0}\right)  ^{\times}\overset{\partial_{v}}{\longrightarrow}\bigoplus
_{Y_{1}}\mathbb{Z}\right)  \text{,}%
\]
where $v$ runs through all discrete rank one valuations of the function field
$k\left(  Y_{0}\right)  $. For a normal integral scheme, we have equality of
rings%
\[
\mathcal{O}_{X}(U)=\bigcap_{v}\mathcal{O}_{v}\subseteq k\left(  Y_{0}\right)
\text{,}%
\]
where $v$ runs through all rank one valuations. In particular, units in
$\mathcal{O}_{X}(U)^{\times}$ are units in all $\mathcal{O}_{v}$, and this
amounts to $\partial_{v}(x)=0$ for all $v$.
\end{proof}

One can widely generalize Lemma \ref{lemma_IsoOfSheaves} to all $\mathcal{K}%
_{n}^{\operatorname*{M}}$, \cite{MR2461425}. While aesthetically very
pleasant, this is not of any use for our immediate purposes.

\subsection{Intersection form -- reprise}

We may now explain how to rephrase the intersection form. Recall from
\S \ref{sect_IntersectionForm} that Lazarsfeld uses the intersection form%
\begin{align}
&  H^{1}(X,\mathcal{O}_{X}^{\times})\otimes\cdots\otimes H^{1}(X,\mathcal{O}%
_{X}^{\times})\nonumber\\
&  \qquad\qquad\longrightarrow H^{2}(X,\mathbb{Z})\otimes\cdots\otimes
H^{2}(X,\mathbb{Z})\overset{\smile}{\longrightarrow}H^{2n}(X,\mathbb{Z}%
)\overset{\cdot\lbrack X]}{\longrightarrow}\mathbb{Z}\text{.} \label{lta2}%
\end{align}
on a complex variety. Here the first arrow stems from the first Chern class
map $c_{1}:H^{1}(X,\mathcal{O}_{X}^{\times})\rightarrow H^{2}(X,\mathbb{Z})$
from the exponential sequence, Equation \ref{lta1}, applied to each tensor
slot individually. On the derived level, we can write the latter as a map%
\begin{equation}
\mathcal{O}_{X}^{\times}\overset{c_{1}}{\longrightarrow}\mathbb{Z}[1]\text{,}
\label{lta2a}%
\end{equation}
where $\mathbb{Z}[1]$ is a complex concentrated in degree one.

\begin{remark}
More precisely, the exact sequence in Equation \ref{lta1} defines an extension
class in the group $\operatorname*{Ext}^{1}(\mathcal{O}_{X}^{\times
},\mathbb{Z})$ in the category of sheaves with values in abelian groups. In
the derived category this $\operatorname*{Ext}$-group gets interpreted as
$\operatorname*{Ext}^{m}(\mathcal{O}_{X}^{\times},\mathbb{Z})=\mathbf{R}%
\operatorname*{Hom}(\mathcal{O}_{X}^{\times},\mathbb{Z}[m])$, giving Equation
\ref{lta2a}.
\end{remark}

Note that under the tensor product of complexes, this means that
$\mathbb{Z}[p]\otimes\mathbb{Z}[q]\cong\mathbb{Z}[p+q]$ (which one sees by
taking the total complex of the tensor product complex). Thus, from this
angle, the first arrow in Equation \ref{lta2} really stems from%
\begin{equation}
\mathcal{O}_{X}^{\times}\otimes\cdots\otimes\mathcal{O}_{X}^{\times
}\longrightarrow\mathbb{Z}[1]\otimes\cdots\otimes\mathbb{Z}[1]\overset{\sim
}{\longrightarrow}\mathbb{Z}[n] \label{lta2b}%
\end{equation}
and in cohomology this induces a map%
\[
H^{i}(X,\mathcal{O}_{X}^{\times}\otimes\cdots\otimes\mathcal{O}_{X}^{\times
})\longrightarrow H^{i}(X,\mathbb{Z}[n])\overset{\sim}{\longrightarrow}%
H^{i+n}(X,\mathbb{Z})\text{.}%
\]
Thus, we observe that the construction used in Lazarsfeld's book
\cite{MR2095471} factors over the top row in the diagram%
\begin{equation}%
\xymatrix{
H^{1}(X,\mathcal{O}_{X}^{\times})\otimes\cdots\otimes H^{1}(X,\mathcal{O}%
_{X}^{\times}) \ar[r] &
H^{n}(X,\mathcal{O}_{X}^{\times} \otimes\cdots\otimes\mathcal{O}_{X}^{\times
}) \ar[r] \ar[d]_{\gamma} &
H^{2n}(X,\mathbb{Z}) \ar[r] &
\mathbb{Z} \\
& H^{n}(X,\mathcal{K}_n^{\operatorname{M}}). \ar[urr]_{\tau}
}
\label{lfe1}%
\end{equation}
The map $\gamma$ is induced from the isomorphism of sheaves $\mathcal{K}%
_{1}^{\operatorname*{M}}\cong\mathcal{O}_{X}^{\times}$ (see Lemma
\ref{lemma_IsoOfSheaves}) and then uses the graded-commutative product
structure of the sheaf $\mathcal{K}_{p}^{\operatorname*{M}}\otimes
\mathcal{K}_{q}^{\operatorname*{M}}\longrightarrow\mathcal{K}_{p+q}%
^{\operatorname*{M}}$ to get to $\mathcal{K}_{n}^{\operatorname*{M}}$. Note
that this map $\gamma$ replaces the use of the first Chern class in Equation
\ref{lta2b}.

What is not at all obvious, is that the factorization $\tau$ exists, and this
is where Theorem \ref{thm_BQ} enters. However, for our purposes we only really
need an explicit formula for $\tau$:

\begin{theorem}
\label{thm_IntersectFormula}Let $X/k$ be an $n$-dimensional irreducible smooth
proper variety. Let $\mathfrak{U}=(U_{i})_{i\in I}$ be an open cover where $I$
is some index set. Suppose%
\begin{equation}
X=\left.  \overset{\cdot}{\bigcup_{\alpha\in I}}\Sigma_{\alpha}\right.
\qquad\text{with}\qquad\Sigma_{\alpha}\subseteq U_{\alpha} \label{ltaps7}%
\end{equation}
is a disjoint decomposition as a set. For $p\in\{1,\ldots,n\}$ let $L_{p}$ be
a line bundle on $X$, and suppose it can locally be trivialized on the open
cover $\mathfrak{U}$, so that its $1$-cocycle takes the form%
\[
(f_{\alpha,\beta}^{p})_{\alpha,\beta\in I}\in\check{H}^{1}(\mathfrak{U}%
,\mathcal{O}_{X}^{\times})
\]
in \v{C}ech cohomology. Then%
\[
\int_{X}c_{1}(L_{1})\cdots c_{1}(L_{n})=\sum_{Y_{\bullet}}[k\left(
Y_{n}\right)  :k]\partial_{Y_{n}}^{Y_{n-1}}\cdots\partial_{Y_{1}}^{Y_{0}%
}\{f_{\alpha(Y_{0})\alpha(Y_{1})}^{1},f_{\alpha(Y_{1})\alpha(Y_{2})}%
^{2},\ldots,f_{\alpha(Y_{n-1})\alpha(Y_{n})}^{n}\}\text{,}%
\]
where

\begin{enumerate}
\item the sum runs over all flags%
\[
Y_{\bullet}:Y_{n}\subset\cdots\subset Y_{1}\subset Y_{0}%
\]
in the sense of Definition \ref{def_Flag} (i.e. they need not be admissible!);

\item for any integral closed subscheme $Y\hookrightarrow X$ let $\alpha(Y)$
denote the unique index $\alpha$ such that the generic point of $Y$ lies in
$\Sigma_{\alpha}$ (see\ Equation \ref{ltaps7});

\item $k\left(  Y_{n}\right)  $ denotes the function field of the closed point
$Y_{n}$ (i.e. it is the residue field of this scheme point).
\end{enumerate}

On the right-hand side there exist only finitely many flags $Y_{\bullet}$ such
that the summand is non-zero.
\end{theorem}

This is \cite[Proposition 3]{MR3104562}. See the rest of the cited paper for
more background on how this formula for intersection multiplicities works.

\begin{remark}
\label{rmk_L4}A few remarks on the intersection theory being used: If
$D_{1},\ldots,D_{n}$ are divisors such that $\bigcap D_{i}=pt$ is a single
closed point, the local intersection index can be computed via%
\begin{equation}
\operatorname*{res}\nolimits_{pt}\frac{\mathrm{d}f_{1}}{f_{1}}\wedge
\cdots\wedge\frac{\mathrm{d}f_{n}}{f_{n}}\text{,}\label{lhi1}%
\end{equation}
where $f_{i}$ is a local equation for $D_{i}$ and $\operatorname*{res}%
\nolimits_{pt}$ is a somewhat shortened abbreviation for Grothendieck's
residue symbol, see Griffiths--Harris \cite[p. 663, and end of page
669]{MR1288523}. Any such expression can be rephrased in terms of rank $n$
valuation vectors, by a variant of Proposition \ref{prop_DetFormula} below for
residues. However, for ample divisors, the volume is related to the top
self-intersection number, so all $D_{i}$ agree, and the condition $\bigcap
D_{i}=pt$ fails \textit{miserably} (and even if it did not, the $n$-form in
Equation \ref{lhi1} is zero). One could fix this issue by moving each $D_{i}$
within its linear equivalence class, rather non-canonically. However, the
formulation of intersection numbers as above does not run into this problem at
all. It would turn out to express intersection numbers in terms of rank $n$
valuation vectors, irrespective of the dimension of the set-theoretic
intersection, in fact this would even work if the self-intersection was highly negative.
\end{remark}

\section{The main theorem}

\subsection{Top rank valuations and $n$-DVFs}

\begin{definition}
We define an $n$\emph{-DVF with last residue field }$\kappa$ (short for
\textquotedblleft$n$-discrete valuation field\textquotedblright) as follows:

\begin{itemize}
\item A $0$-DVF with last residue field $\kappa$ is the field $\kappa$ itself.

\item An $n$-DVF\ with last residue field $\kappa$ is a field $(F,v)$ with a
discrete rank one valuation $v$ such that $\kappa(v)$ is an $(n-1)$-DVF with
last residue field $\kappa$.
\end{itemize}
\end{definition}

Although it is cleanest to use this inductive definition, one may visualize an
$n$-DVF $F$ as the following structure%
\begin{equation}%
\bfig\node x(0,1200)[F]
\node y(0,900)[\mathcal{O}_{v_1}]
\node z(300,900)[k_1]
\node w(300,600)[\mathcal{O}_{v_2}]
\node u(600,600)[k_2,]
\node v(600,300)[\vdots]
\arrow/{^{(}->}/[y`x;]
\arrow/{->>}/[y`z;]
\arrow/{^{(}->}/[w`z;]
\arrow/{->>}/[w`u;]
\arrow[v`u;]
\efig
\label{lt4b}%
\end{equation}
where each $\mathcal{O}_{v_{i}}$ is the valuation ring for the discrete rank
one valuation $v_{i}$ defined on the field displayed above $\mathcal{O}%
_{v_{i}}$ (so that the upward arrows are the inclusions of the valuation
rings), while each surjection to the right is the quotient map $\mathcal{O}%
_{v_{i}}\twoheadrightarrow\mathcal{O}_{v_{i}}/\mathfrak{m}_{v_{i}}=:k_{i}$ to
the respective residue field.

Suppose $X/k$ is an irreducible $n$-dimensional smooth projective variety. Let
$k\left(  X\right)  $ be its function field. If $\underline{v}:k\left(
X\right)  ^{\times}\rightarrow\mathbb{Z}_{\operatorname*{lex}}^{n}$ is a top
rank valuation, it takes values in $\mathbb{Z}_{\operatorname*{lex}}^{n}$,
lexicographically ordered, and one can write it as $\underline{v}%
(f)=(\underline{v}_{1}(f),\ldots,\underline{v}_{n}(f))$, where $\underline
{v}_{i}(f):k\left(  X\right)  ^{\times}\rightarrow\mathbb{Z}$ is our notation
for the components of this valuation vector.

\begin{remark}
This construction would make sense for any field instead of \textquotedblleft%
$k\left(  X\right)  $\textquotedblright, but we will only use it in the said context.
\end{remark}

It turns out that $\underline{v}_{1}$ is a rank one $\mathbb{Z}$-valued
valuation on $k\left(  X\right)  ^{\times}$, defining a valuation ring
$\mathcal{O}_{v_{1}}$; and this can be continued interatively along zig-zag as
in Figure \ref{lt4b}. So, every top rank valuation $\underline{v}:k\left(
X\right)  ^{\times}\longrightarrow\mathbb{Z}_{\operatorname*{lex}}^{n}$
determines the structure of an $n$-DVF with last residue field being the
residue field of $\underline{v}$.

\begin{example}
Every admissible flag in the sense of Definition \ref{def_Flag} defines a top
rank valuation. If we allow ourselves to switch to a different birational
model, any finite number of top rank valuations can be made to come from
admissible flags, \cite[Theorem 2.9]{MR3694973}.
\end{example}

Conversely, giving the structure of an $n$-DVF on $k\left(  X\right)  $, then
if we addtionally pick uniformizers, it pins down a unique valuation
$\underline{v}:k\left(  X\right)  ^{\times}\longrightarrow\mathbb{Z}%
_{\operatorname*{lex}}^{n}$:

\begin{definition}
\label{def_HRankVal}Let $F$ be an $n$-DVF. Fix uniformizers $\pi_{i}%
\in\mathcal{O}_{v_{i}}$. Define a rank $n$ valuation%
\begin{align}
\underline{v}:F^{\times}  &  \longrightarrow\mathbb{Z}_{\operatorname*{lex}%
}^{n}\label{lt4}\\
f  &  \longmapsto\left(  v_{1}(f),v_{2}\left(  \overline{f\pi_{1}^{-v_{1}(f)}%
}\right)  ,v_{3}\left(  \overline{\overline{f\pi_{1}^{-v_{1}(f)}}\pi
_{2}^{-v_{2}\left(  \overline{f\pi_{1}^{-v_{1}(f)}}\right)  }}\right)
,\ldots\right)  \text{.}\nonumber
\end{align}
For $n\geq2$ this valuation depends on the choice of the uniformizers.
\end{definition}

We have intentionally mimicked the definitions in \cite[\S 1.1]{MR2571958} to
stress the analogy.

\begin{remark}
Let us explain the idea and also give a more precise definition: The first
entry is just the valuation $v_{1}$ itself. Next, note that $f\pi_{1}%
^{-v_{1}(f)}$ by construction has valuation zero, so it lies in $\mathcal{O}%
_{v_{1}}^{\times}$ and thus its image under the quotient map to the residue
field is also invertible, i.e.%
\[
\overline{f\pi_{1}^{-v_{1}(f)}}\in k_{1}^{\times}\text{.}%
\]
Hence, it makes sense to consider its valuation with respect to $v_{2}$. Now
repeat this. That is, we multiply with an appropriate power of $\pi_{2}$ to
ensure we get an element in $\mathcal{O}_{v_{2}}^{\times}$. It is more elegant
to define the higher rank valuation by inductively defining%
\begin{equation}
f_{1}:=f\qquad\text{and}\qquad f_{p+1}:=\overline{f_{p}\pi_{p}^{-v_{p}(f_{p}%
)}} \label{lt6}%
\end{equation}
for $p\geq1$ and $f_{p+1}\in k_{p}$. Then%
\begin{equation}
f\mapsto(v_{1}(f_{1}),v_{2}(f_{2}),\ldots,v_{n}(f_{n}))\text{.} \label{lt6y}%
\end{equation}
Unravelling the inductive nature of this definition, we find Equation
\ref{lt4}.
\end{remark}

\begin{example}
Consider $X:=\mathbb{A}_{k}^{2}=\operatorname*{Spec}k[s,t]$. It has a $2$-DVF
structure coming from the flag of ideals $(s,t)\subset(t)\subset(0)$,%
\[%
\bfig\node x(0,1200)[k(s,t)]
\node y(0,900)[k(s){[t]}_{(t)}]
\node z(600,900)[k(s)]
\node w(600,600)[k{[s]}_{(s)}]
\node u(1100,600)[k.]
\arrow/{^{(}->}/[y`x;]
\arrow/{->>}/[y`z;]
\arrow/{^{(}->}/[w`z;]
\arrow/{->>}/[w`u;]
\efig
\]
Then for any $m\in\mathbb{Z}$, $\pi_{1}:=s^{-m}t$ is a uniformizer for
$k(s)[t]_{(t)}$. We obtain the valuation vector $\underline{v}(t)=(1,m)$,
depending on $m$. This cannot happen for the valuation vectors in the
definition of Newton--Okounkov bodies. The whole point is that $\pi
_{1}:=s^{-m}t$ only lies in the local ring $k[s,t]_{(s,t)}$ of the variety if
$m\leq0$, and then is only a local equation for $(t)$ if $m=0$, for otherwise
it is a local equation for a two component subvariety, with components cut out
by $(t)$ and a nil-thickening of $(s)$. Thus, only $m=0$ is possible in the
setting of Okounkov bodies. For a general $n$-DVF, there is no analogue of the
local ring which allows us to single out such problematic choices of uniformizers.
\end{example}

Let us point out that one can define the valuation vectors which appear in the
construction of Newton--Okounkov bodies entirely in the language of Milnor $K$-groups:

\begin{lemma}
\label{lemma_RewriteForMilnor}Let $X/k$ be an irreducible smooth proper
variety of dimension $n$ and $\underline{v}$ the top rank valuation coming
from an admissible flag $Y_{\bullet}$ in $X$. Let $f\in k\left(  X\right)
^{\times}$ be given. Let $\pi_{i}$ be the local equations of the individual
flag members. Define a vector $\underline{x}\in\mathbb{Z}^{n}$ with components%
\[
\underline{x}_{p}:=\partial_{v_{p}}\circ\cdots\circ\partial_{v_{1}}\{\pi
_{1},\ldots,\pi_{p-1},f\}\text{.}%
\]
Then $\underline{v}(f):=(\underline{x}_{1},\ldots,\underline{x}_{n})$.
\end{lemma}

Note that our choice of the $\pi_{i}$ is more strict that just requiring them
to be arbitrary uniformizers in $\mathcal{O}_{v_{i}}$; they all lie in the
local ring $\mathcal{O}_{X,Y_{n}}$. One could rewrite the vector as%
\[
\underline{x}_{1}=\partial_{v_{1}}\{f\}\qquad\underline{x}_{p}:=\partial
_{v_{p}}s_{v_{p-1}}^{-\pi_{p-1}}\cdots s_{v_{1}}^{-\pi_{1}}\{f\}
\]
for $p\geq2$ using the specialization maps of Equation \ref{ltya2}.

\begin{proof}
Since all valuation vector components lie in $\mathbb{Z}$, we can throughout
the proof neglect elements which are $2$-torsion, for once we map our
constructs to $\mathbb{Z}$, $2$-torsion elements must necessarily go to zero.
Write $f=u_{m}\pi_{m}^{v_{m}(f)}$ and then for any $m$ such that $1\leq m\leq
p$, we compute%
\begin{align}
&  \partial_{v_{p}}\circ\cdots\circ\partial_{v_{m}}\{\pi_{m},\ldots,\pi
_{p-1},f\}\nonumber\\
&  \qquad=\partial_{v_{p}}\circ\cdots\circ\partial_{v_{m}}\{\pi_{m},\ldots
,\pi_{p-1},u_{m}\}+v_{m}(f)\partial_{v_{p}}\circ\cdots\circ\partial_{v_{m}%
}\{\pi_{m},\ldots,\pi_{p-1},\pi_{m}\}\nonumber\\
&  \qquad=\partial_{v_{p}}\circ\cdots\circ\partial_{v_{m+1}}\{\overline
{\pi_{m+1}},\ldots,\overline{\pi_{p-1}},\overline{f\pi_{m}^{-v_{m}(f)}%
}\}+\text{(}2\text{-torsion)} \label{lta5y}%
\end{align}
by using multi-linearity, Lemma \ref{lemma_GradedCommutative} and Equation
\ref{lt3}. For using the latter, we have used that since the $\pi_{i}$ are
elements in the local equations for the flag in the local ring of $X$ at
$Y_{0}$, their image is a uniformizer in $\mathcal{O}_{v_{i}}$, but
necessarily a unit in all $\mathcal{O}_{v_{j}}$ with $j\neq i$. Now we observe
that the above equation admits an induction along the variable $m$. Starting
with $m=1$, the last slot \textquotedblleft$\overline{f\pi_{m}^{-v_{m}(f)}}%
$\textquotedblright\ in Equation \ref{lta5y} follows the same inductive
description along $m$ as in Equation \ref{lt6}. Since both the induction loc.
cit. as well as here begin with $f$, it follows that once $m$ equals $p$, we
will have arrived at%
\[
\underline{x}_{p}=_{\operatorname*{def}}\partial_{v_{p}}\circ\cdots
\circ\partial_{v_{1}}\{\pi_{1},\ldots,\pi_{p-1},f\}=\partial_{v_{p}}%
\{f_{p}\}+\text{(}2\text{-torsion),}%
\]
with the same meaning of $f_{p}$ as in Equation \ref{lt6}. This is just
$v_{p}(f_{p})$, so we obtain that the vector $\underline{x}$ agrees with the
one in Equation \ref{lt6y}. This proves our claim.
\end{proof}

Note that this gives a reformulation of the definition of Okounkov bodies
which does not involve rank $\geq2$ valuations anymore. Of course, apart from
that, it is the same as the usual one.

\begin{definition}
[Milnor-$K$-style Okounkov bodies]\label{def_DetOkounkovBodies}Let $X/k$ be an
irreducible smooth proper variety of dimension $n$ and $\underline{v}$ the top
rank valuation coming from an admissible flag $Y_{\bullet}$ in $X$. Then the
Newton--Okounkov body can be rephrased \textquotedblleft for Milnor $K$-theory
aficionados\textquotedblright\ as the closed convex hull of the vectors%
\[
\left\{  \frac{1}{m}\left(  \partial_{v_{1}}\{f\},\ldots,\partial_{v_{n}}%
\circ\cdots\circ\partial_{v_{1}}\{\pi_{1},\ldots,\pi_{n-1},f\}\right)
\right\}  _{f,m}%
\]
where $(f,m)$ runs through all pairs $m\geq1$ and $f\in H^{0}(X,\mathcal{O}%
_{X}(mD))\setminus\{0\}$. The $\pi_{i}$ are arbitrary local equations of the
flag entries and $v_{i}$ the rank one valuations arising from $\underline{v}$.
\end{definition}

\subsection{Local computations}

The following is surely well-known to experts, but since it is quite important
for our results, we provide full details. The same type of determinant
evaluation is used in \cite[\S 2.2]{MR697316}.

\begin{proposition}
\label{prop_DetFormula}Let $F$ be an $n$-DVF. Fix uniformizers $\pi_{i}%
\in\mathcal{O}_{v_{i}}$. Suppose $f_{1},\ldots,f_{n}\in F^{\times}$. Then%
\begin{equation}
(\partial_{v_{n}}\cdots\partial_{v_{1}})\{f_{1},\ldots,f_{n}\}=\det%
\begin{pmatrix}
\underline{v}_{1}(f_{1}) & \underline{v}_{2}(f_{1}) & \cdots & \underline
{v}_{n}(f_{1})\\
\underline{v}_{1}(f_{2}) & \underline{v}_{2}(f_{2}) &  & \vdots\\
\vdots &  & \ddots & \\
\underline{v}_{1}(f_{n}) & \cdots &  & \underline{v}_{n}(f_{n})
\end{pmatrix}
\text{,} \label{lt5}%
\end{equation}
where $\underline{v}(f)=(\underline{v}_{1}(f),\ldots,\underline{v}_{n}(f))$
encodes the components of the rank $n$ valuation attached to $F$ via
Definition \ref{def_HRankVal}. On the left-hand side the $v_{1},\ldots,v_{n}$
denote the discrete rank one valuations of the $n$-DVF\ structure attached to
$\underline{v}$, i.e. the valuations which appear in Figure \ref{lt4b}.
\end{proposition}

\begin{proof}
We prove this by induction on $n$. For $n=1$ the claim just reads%
\[
\partial_{v_{1}}\{f_{1}\}=\underline{v}_{1}(f_{1})=v_{1}(f_{1})
\]
and thus reduces to the explanation below Equation \ref{lt3a}. Suppose the
proposition is proven for all $\ell$-DVFs with $\ell<n$. Then we prove the
case $n$: (Step 1) Both the left-hand side as well as the right-hand side in
Equation \ref{lt5} change sign under swapping $f_{i}$ with $f_{j}$: On the
left side, this is the graded-commutativity of Milnor $K$-theory, i.e. Lemma
\ref{lemma_GradedCommutative}. On the right side it amounts to swapping two
rows, so it is a standard property of the determinant. Furthermore, both the
left and right side are multiplicative in each slot: On the left side, this is
already true for the tensor algebra $TF^{\times}$. On the right side, this
follows from the multiplicativity of valuations. Next, we claim that both
sides vanish if two slots agree, i.e. $f_{i}=f_{j}$ for some $i\neq j$. This
is clear: Under anti-commutativity under swapping the two, it follows that
both sides must be $2$-torsion. However, both sides take values in
$\mathbb{Z}$. Next, we claim that it suffices to show that left and right side
agree on all elements of the shape%
\[
\text{(A) }\{u_{1},u_{2},\ldots,u_{n}\}\qquad\qquad\text{and}\qquad
\qquad\text{(B) }\{\pi,u_{2},\ldots,u_{n}\}\text{,}%
\]
where $\pi$ is a uniformizer for $\mathcal{O}_{v_{1}}$, i.e. $v_{1}(\pi)=1$,
and $u_{1},\ldots,u_{n}\in\mathcal{O}_{v_{1}}^{\times}$. To see this, use
exactly the same argument as in the proof of Lemma
\ref{lemma_StdShapeForValuation}, just replace the relation $\{\pi,\pi
\}=\{\pi,-1\}$ by using that once two slots agree, both sides are already
zero.\newline(Step 2) For the elements of shape (A) we find $\partial_{v_{1}%
}\{u_{1},\ldots,u_{n}\}=0$ by Equation \ref{lt3b}. On the other hand, in the
matrix on the right side in Equation \ref{lt5} the first colum reads%
\[%
\begin{pmatrix}
\underline{v}_{1}(u_{1})\\
\underline{v}_{1}(u_{2})\\
\vdots\\
\underline{v}_{1}(u_{n})
\end{pmatrix}
=%
\begin{pmatrix}
0\\
0\\
\vdots\\
0
\end{pmatrix}
\text{,}%
\]
so the determinant also vanishes. We find $\partial_{v_{1}}\{\pi,u_{2}%
,\ldots,u_{n}\}=\{\overline{u_{2}},\ldots,\overline{u_{n}}\}$ for elements of
shape (B) by Equation \ref{lt3}. On the other hand, we claim that%
\begin{equation}%
\begin{pmatrix}
\underline{v}_{1}(\pi) & \underline{v}_{2}(\pi) & \cdots & \underline{v}%
_{n}(\pi)\\
\underline{v}_{1}(u_{2}) & \underline{v}_{2}(u_{2}) &  & \underline{v}%
_{n}(u_{2})\\
\vdots &  & \ddots & \vdots\\
\underline{v}_{1}(u_{n}) & \cdots &  & \underline{v}_{n}(u_{n})
\end{pmatrix}
=%
\begin{pmatrix}
1 & 0 & \cdots & 0\\
0 & \underline{v}_{2}(u_{2}) &  & \underline{v}_{n}(u_{2})\\
\vdots &  & \ddots & \vdots\\
0 & \underline{v}_{2}(u_{n}) & \cdots & \underline{v}_{n}(u_{n})
\end{pmatrix}
\label{lt7}%
\end{equation}
We check the top row: Of course we have $\underline{v}_{1}(\pi)=v_{1}(\pi)=1$.
This holds for any uniformizer, so we may use $\pi:=\pi_{1}$. Thus, in terms
of Equation \ref{lt6} for $f:=\pi$ we have $f_{1}=\pi$ and then $f_{2}%
:=\overline{\pi\pi^{-v_{1}(\pi)}}=1$. But if $f_{p}=1$ for any $p\geq2$, we
inductively find $f_{p+1}:=\overline{1\cdot\pi_{p}^{-v_{p}(1)}}=1$. Checking
the left column in Equation \ref{lt7} just amounts to observing that
$\underline{v}_{1}(u_{i})=v_{1}(u_{i})=1$ since $u_{i}\in\mathcal{O}_{v_{1}%
}^{\times}$. Thus,%
\begin{equation}
\det%
\begin{pmatrix}
\underline{v}_{1}(\pi) & \underline{v}_{2}(\pi) & \cdots & \underline{v}%
_{n}(\pi)\\
\underline{v}_{1}(u_{2}) & \underline{v}_{2}(u_{2}) &  & \underline{v}%
_{n}(u_{2})\\
\vdots &  & \ddots & \vdots\\
\underline{v}_{1}(u_{n}) & \cdots &  & \underline{v}_{n}(u_{n})
\end{pmatrix}
=\det%
\begin{pmatrix}
\underline{v}_{2}(u_{2}) &  & \underline{v}_{n}(u_{2})\\
& \ddots & \vdots\\
\underline{v}_{2}(u_{n}) & \cdots & \underline{v}_{n}(u_{n})
\end{pmatrix}
\text{.} \label{lt7a}%
\end{equation}
We may now use that the residue field $k_{1}$ (as in Diagram \ref{lt4b}) is an
$(n-1)$-DVF. Since $u_{2},\ldots,u_{n}\in\mathcal{O}_{v_{1}}^{\times}$ the
rank $(n-1)$ valuation $\underline{v}_{[k_{1}]}$ of $\overline{u_{i}}$ with
respect to $k_{1}$ agrees with the last $n-1$ entries of the rank $n$
valuation $\underline{v}$ which appears on the right in Equation \ref{lt7a}.
To see this, note that if we start the inductive definition of the valuation
vector in Equation \ref{lt6}, its second term is%
\[
f_{1}:=u_{i}\qquad\qquad\qquad f_{2}:=\overline{u_{i}\pi_{p}^{-v_{1}(u_{i})}%
}=\overline{u_{i}}%
\]
and then the next inductive steps towards $f_{3},f_{4},\ldots$ agree with
whether we perform them with respect to $F$ or with respect to $k_{1}$, just
differing by an indexing shift. By our inductive hypothesis, the proposition
is already proven for all $(n-1)$-DVFs, so we obtain.%
\[
(\partial_{v_{n}}\cdots\partial_{v_{2}})\{\overline{u_{2}},\ldots
,\overline{u_{n}}\}=\det%
\begin{pmatrix}
\underline{v}_{2}(u_{2}) &  & \underline{v}_{n}(u_{2})\\
& \ddots & \vdots\\
\underline{v}_{2}(u_{n}) & \cdots & \underline{v}_{n}(u_{n})
\end{pmatrix}
\text{.}%
\]
Thus, $(\partial_{v_{n}}\cdots\partial_{v_{1}})\{\pi,u_{2},\ldots
,u_{n}\}=(\partial_{v_{n}}\cdots\partial_{v_{2}})\{\overline{u_{2}}%
,\ldots,\overline{u_{n}}\}$ agrees with the right hand side. This finishes the proof.
\end{proof}

\begin{corollary}
\label{cor_IndepUnif}The determinant on the right-hand side is independent of
the choice of uniformizers $\pi_{i}\in\mathcal{O}_{v_{i}}$.
\end{corollary}

Corollary \ref{cor_IndepUnif} follows immediately from Proposition
\ref{prop_DetFormula} since the left-hand side is independent of the choice of uniformizers.

\begin{definition}
\label{def_SetG}Let $X/k$ be an irreducible smooth proper variety of dimension
$n$. Let $Y_{\bullet}$ be a flag in $X$ in the sense of Definition
\ref{def_Flag}, not necessarily admissible. Let%
\[
\mathcal{G}(Y_{\bullet})
\]
be the set of all $n$-DVF structures on the function field $k\left(  X\right)
$%
\begin{equation}%
\bfig\node x(0,1200)[{k\left( X\right)}]
\node y(0,900)[\mathcal{O}_{w_1}]
\node z(300,900)[k_1]
\node w(300,600)[\mathcal{O}_{w_2}]
\node u(600,600)[k_2,]
\node v(600,300)[\vdots]
\arrow/{^{(}->}/[y`x;]
\arrow/{->>}/[y`z;]
\arrow/{^{(}->}/[w`z;]
\arrow/{->>}/[w`u;]
\arrow[v`u;]
\efig
\label{lz2}%
\end{equation}
with the following properties:

\begin{enumerate}
\item The first valuation ring $\mathcal{O}_{w_{1}}\subseteq k\left(
X\right)  $ satisfies%
\[
\mathcal{O}_{X,Y_{1}}\subseteq\mathcal{O}_{w_{1}}\subseteq k\left(  X\right)
\]
and is of the following form: The ring $\mathcal{O}_{w_{1}}$ is the
localization of the integral closure of the left ring inside $k\left(
X\right)  $ at any of its maximal ideals.

\item For $i\geq2$, the $i$-th valuation ring in\ Figure \ref{lz2} satisfies%
\[
\mathcal{O}_{Y_{i-1},Y_{i}}\subseteq\mathcal{O}_{w_{i}}\subseteq\kappa
(w_{i-1})
\]
and is of the following form: The ring $\mathcal{O}_{w_{i}}$ is the
localization of the integral closure of the left ring inside $\kappa(w_{i-1})$
at any of its maximal ideals.
\end{enumerate}
\end{definition}

\begin{remark}
\label{rmk_GSetTrivialForAdmissibleFlag}If $Y_{\bullet}$ happens to be an
admissible flag, all the local rings $\mathcal{O}_{Y_{i-1},Y_{i}}$ are
localizations of $\mathcal{O}_{Y_{i-1},Y_{0}}$ and therefore regular. Hence,
in this case the set $\mathcal{G}(Y_{\bullet})$ contains only a single element.
\end{remark}

\begin{remark}
By the finiteness of integral closure under these assumptions (essentially
finite type over a field), it follows that $\mathcal{G}(Y_{\bullet})$ is
always a finite set.
\end{remark}

\begin{example}
Let $C$ be the nodal cubic with equation $y^{2}-x^{2}(x+1)=0$ in
$\mathbb{A}_{\mathbb{C}}^{2}$. Then the flag $Y_{\bullet}:\mathbb{A}%
_{\mathbb{C}}^{2}\supset C\supset(0,0)$ in $\mathbb{A}_{\mathbb{C}}^{2}$ is
not admissible. The set $\mathcal{G}(Y_{\bullet})$ will consist of two
elements. In the usual embedded resolution $R$ of $C\hookrightarrow
\mathbb{A}_{\mathbb{C}}^{2}$, let $\tilde{C}$ be the resolved curve (to get
this resolution, just blow up the origin). Here $\tilde{C}\rightarrow C$ is
the normalization. In the resolution the point $(0,0)$ will have two preimages
$p_{1},p_{2}$. The admissible flags $R\supset\tilde{C}\supset p_{i}$ with
$i=1,2$ in the resolution $R$ give rise to the $2$-DVF structures on $k\left(
X\right)  $ in $\mathcal{G}(Y_{\bullet})$. This is probably the simplest
example with cardinality $\#\mathcal{G}(Y_{\bullet})>1$.
\end{example}

\subsection{Global computations}

\begin{theorem}
\label{thm_MainFormula}Let $X/k$ be an irreducible smooth projective variety
of dimension $n$. Let $\mathfrak{U}=(U_{\alpha})_{\alpha\in I}$ be a finite
open cover in the Zariski topology. Suppose%
\begin{equation}
X=\overset{\cdot}{\bigcup_{\alpha\in I}}\Sigma_{\alpha}\qquad\text{with}%
\qquad\Sigma_{\alpha}\subseteq U_{\alpha} \label{ltt0b}%
\end{equation}
is a disjoint decomposition as a set. Let $D$ be an ample divisor and suppose
that $\mathcal{O}_{X}(D)$ is locally given by the \v{C}ech cocycle
$(h_{\alpha})_{\alpha\in I}$ in $\check{H}^{1}(\mathfrak{U},\mathcal{K}%
_{X}^{\times}/\mathcal{O}_{X}^{\times})$. Let $\underline{v}:k\left(
X\right)  ^{\times}\rightarrow\mathbb{Z}_{\operatorname*{lex}}^{n}$ be a top
rank valuation. Then the volume of the Newton--Okounkov body is given by the
formula%
\begin{align*}
\operatorname*{Vol}\Delta_{\underline{v}}(D)  &  =\sum_{Y_{\bullet}}%
\sum_{(w_{n},\ldots,w_{1})\in\mathcal{G}(Y_{\bullet})}\sum_{c=0}^{n}%
(-1)^{c}[k\left(  w_{n}\right)  :k]\\
&  \qquad\frac{1}{n!}\det%
\begin{pmatrix}
\underline{w}_{1}(h_{\alpha(Y_{0})}) & \cdots & \widehat{\underline{w}%
_{1}(h_{\alpha(Y_{c})})} & \cdots & \underline{w}_{1}(h_{\alpha(Y_{n})})\\
\underline{w}_{2}(h_{\alpha(Y_{0})}) & \ddots &  &  & \underline{w}%
_{2}(h_{\alpha(Y_{n})})\\
\vdots &  &  &  & \vdots\\
\underline{w}_{n}(h_{\alpha(Y_{0})}) & \cdots & \widehat{\underline{w}%
_{n}(h_{\alpha(Y_{c})})} & \cdots & \underline{w}_{n}(h_{\alpha(Y_{n})})
\end{pmatrix}
\text{,}%
\end{align*}
where

\begin{enumerate}
\item $Y_{\bullet}$ runs through all flags in $X$, but only finitely many will
contribute a non-zero summand,

\item the finite set $\mathcal{G}(Y_{\bullet})$ is the one of Definition
\ref{def_SetG},

\item $\alpha(Y)\in I$ denotes the unique index such that the generic point of
$Y$ is contained in $\Sigma_{\alpha}$ of Equation \ref{ltt0b},
\end{enumerate}
\end{theorem}

It may or may not happen that the flag $Y_{\bullet}^{\prime}$ appears among
those $Y_{\bullet}$ with a non-zero contribution to the sum. Since we can
change $Y_{\bullet}^{\prime}$ without changing the \v{C}ech cocycle, one can
easily produce an example for either situation by adapting $Y_{\bullet
}^{\prime}$ as needed.

\begin{proof}
(Step 1) We begin with Theorem A of \cite{MR2571958}, i.e.%
\begin{equation}
\operatorname*{Vol}\Delta_{\underline{v}}(\mathcal{O}_{X}(D))=\lim
_{m\rightarrow\infty}\frac{\dim_{k}H^{0}(X,\mathcal{O}_{X}(mD))}{m^{n}%
}\text{.} \label{ltt5a}%
\end{equation}
In fact, the cited theorem only shows this in the special case where
$\underline{v}$ comes from an admissible flag in the sense of Definition
\ref{def_Flag}. However, it is true for general $\underline{v}$. This is
explained in \cite[Remark 2.12]{MR3694973}: It suffices to work with a
birational modification%
\[
f:\tilde{X}\longrightarrow X\text{,}%
\]
where $\underline{v}$ comes from an admissible flag. Such an $\tilde{X}$
exists by a type of embedded resolution along the centers of the valuation,
see \cite[Theorem 2.9]{MR3694973} for the precise result. Finally, the term on
the right in Equation \ref{ltt5a} is a birational invariant, i.e.
$H^{0}(\tilde{X},f^{\ast}\mathcal{O}_{X}(mD))$ has the same growth,
\cite[Proposition 2.2.43]{MR2095471}, and here Theorem A of \cite{MR2571958}
applies verbatim. Note that the definition of $\Delta_{\underline{v}}$ only
depends on the valuation on the function field, so it does not see the change
of the birational model. As $D$ is ample, the asymptotic Riemann--Roch theorem
implies that%
\[
\chi(\mathcal{O}_{X}(mD))=\frac{m^{n}}{n!}\int_{X}c_{1}(\mathcal{O}%
_{X}(D))^{n}+\mathsf{O}(m^{n-1})\text{,}%
\]
see \cite[Theorem 1.1.24]{MR2095471}. By Serre vanishing for ample line
bundles for $m\gg0$ sufficiently big, the higher cohomology groups in the
Euler characteristic vanish, so in the limit in Equation \ref{ltt5a}
simplifies to%
\[
\operatorname*{Vol}\Delta_{\underline{v}}(\mathcal{O}_{X}(D))=\lim
_{m\rightarrow\infty}\frac{\chi(\mathcal{O}_{X}(mD))}{m^{n}}=\frac{1}{n!}%
\int_{X}c_{1}(\mathcal{O}_{X}(D))^{n}\text{.}%
\]
By the way, instead of the above argument for birational invariance of the
volume, we could alternatively compute the intersection number on the right in
the modification $\tilde{X}$; this circumvents the admissible flag problem in
a different way. The reader may now forget about $\underline{v}$ and
$\tilde{X}$; they will not play a r\^{o}le in the rest of the proof.\newline%
(Step 2) Next, under the connecting map $\delta$ in%
\begin{equation}
H^{0}(X,\mathcal{K}_{X}^{\times})\longrightarrow\underset{(h_{\alpha}%
)_{\alpha}}{H^{0}(X,\mathcal{K}_{X}^{\times}/\mathcal{O}_{X}^{\times}%
)}\overset{\delta}{\longrightarrow}\underset{(f_{\alpha,\beta})_{\alpha,\beta
}}{H^{1}(X,\mathcal{O}_{X})}\longrightarrow0 \label{ltt0a}%
\end{equation}
we get a \v{C}ech $1$-cocycle representative for the isomorphism class of the
invertible sheaf $\mathcal{O}_{X}(D)$ attached to the Cartier divisor; we
write $(f_{\alpha,\beta})_{\alpha,\beta\in I}$. We use Theorem
\ref{thm_IntersectFormula} to compute the top self-intersection number of $D$.
We can use the same open cover $\mathfrak{U}$ and simply let $f_{\alpha,\beta
}^{p}:=f_{\alpha,\beta}$ for $p=1,\ldots,n$, indifferently the value of $p$.
We get $\frac{1}{n!}\int_{X}c_{1}(\mathcal{O}_{X}(D))^{n}=$%
\begin{equation}
=\frac{1}{n!}\sum_{Y_{\bullet}}[k\left(  Y_{n}\right)  :k]\left(
\partial_{Y_{n}}^{Y_{n-1}}\cdots\partial_{Y_{1}}^{Y_{0}}\right)
\{f_{\alpha(Y_{0})\alpha(Y_{1})},f_{\alpha(Y_{1})\alpha(Y_{2})},\ldots
,f_{\alpha(Y_{n-1})\alpha(Y_{n})}\}\text{,} \label{ltt1}%
\end{equation}
where the meaning of $Y_{\bullet}$, $\alpha(-)$ is as in the cited theorem.
Each $\partial_{Y_{p+1}}^{Y_{p}}$ by its construction is a sum%
\[
\partial_{Y_{p+1}}^{Y_{p}}=\sum_{i}N_{\kappa(v_{i})/k\left(  Y_{p+1}\right)
}\circ\partial_{v_{i}}%
\]
for a finite number of discrete valuations $v_{i}$ on $k(Y_{p})$; we refer to
\S \ref{sect_BdyMapsBetweenPoints} where we had recalled this in detail. Since
the relevant valuations and the indexing set for $i$ both depend on $Y_{p}$,
let us store this extra data in the notation and rewrite the above expression
as%
\[
\partial_{Y_{p+1}}^{Y_{p}}=\sum_{i_{p}\in J^{(p)}}N_{\kappa(v_{i_{p}}%
^{(p)})/k\left(  Y_{p+1}\right)  }\circ\partial_{v_{i_{p}}^{(p)}}\text{,}%
\]
where $J^{(p)}$ is a finite index set. We apologize for the heavy notation, we
shall not employ it for long. The term $\partial_{Y_{n}}^{Y_{n-1}}%
\cdots\partial_{Y_{1}}^{Y_{0}}$ in Equation \ref{ltt1} unravels as
\begin{align}
&  \partial_{Y_{n}}^{Y_{n-1}}\cdots\partial_{Y_{1}}^{Y_{0}}=\sum_{i_{n-1}\in
J^{(n-1)}}\cdots\sum_{i_{1}\in J^{(1)}}\sum_{i_{0}\in J^{(0)}}\nonumber\\
&  \qquad\qquad N_{\kappa(v_{i_{n-1}}^{(n-1)})/k\left(  Y_{n}\right)
}\partial_{v_{i_{n-1}}^{(n-1)}}\circ\cdots\circ N_{\kappa(v_{i_{1}}%
^{(1)})/k\left(  Y_{2}\right)  }\partial_{v_{i_{1}}^{(1)}}\circ N_{\kappa
(v_{i_{0}}^{(0)})/k\left(  Y_{1}\right)  }\partial_{v_{i_{0}}^{(0)}}
\label{ltt4}%
\end{align}
or perhaps more briefly%
\[
=\sum_{i_{n-1}\in J^{(n-1)}}\cdots\sum_{i_{1}\in J^{(1)}}\sum_{i_{0}\in
J^{(0)}}\prod_{p=0}^{n-1}N_{\kappa(v_{i_{p}}^{(p)})/k\left(  Y_{p+1}\right)
}\circ\partial_{v_{i_{p}}^{(p)}}%
\]
if we agree to read the product as the (non-commutative) concatenation of
morphisms and unravel it from right to left as $p$ increases. In fact, this is
a little better than being plainly non-commutative since the maps originate
and end in different objects, so in a certain sense no actual ambiguity is
possible.\newline(Step 3)\ The expression in Equation \ref{ltt4} may,
structurally, be summarized as%
\begin{equation}
\underset{n\text{ pairs}}{\underbrace{N\partial N\partial\cdots N\partial}%
}=N\underset{(n-1)\text{ pairs}}{\underbrace{(\partial N)\cdots(\partial
N)(\partial N)}}\partial\text{,} \label{ltt5}%
\end{equation}
i.e. an alternating composition of boundary maps and norm maps. Now, perform
inductively starting from the right the following operation: For each pair
$\partial N$ we use Proposition \ref{prop_NormCompat} and write it as%
\begin{equation}
\partial_{v_{i_{p-1}}^{(p-1)}}N_{\kappa(v_{i_{p}}^{(p)})/k\left(
Y_{p+1}\right)  }=\sum_{w}N_{\kappa(w)/\kappa(v_{i_{p-1}})}\partial
_{w}\text{,} \label{luma8b}%
\end{equation}
where the sum on the right runs through the finitely many extensions of the
valuation $v_{i_{p-1}}^{(p-1)}$ to the finite field extension $\kappa
(v_{i_{p}}^{(p)})$. We neglect that $w$ itself of course again depends on what
$p$ we use on the left side. We arrive at an expression of the shape%
\[
N\underset{(n-1)\text{ pairs}}{\underbrace{(\partial N)\cdots(\partial
N)\left.  \underline{(\partial N)}\right.  }}\partial=\sum N(\partial
N)\cdots(\partial N)(\partial N)\left.  \underline{(N\partial)}\right.
\partial
\]
and let us not make precise what the sum on the right side is summing over (it
would be the $w$), and we have underlined the part of the expression where
something has changed. The two consecutive norm maps can be combined to one,
see Equation \ref{l_NormInTowers}. We get%
\[
=\sum N\underset{(n-2)\text{ pairs}}{\underbrace{(\partial N)\cdots(\partial
N)(\partial N)}}\left.  \underline{(\partial)}\right.  \partial\text{,}%
\]
where the underlined expression is what remains from our modifications. The
underbraced term has exactly the same shape as what we had started from in
Equation \ref{ltt5}. Inductively repeat the procedure, working from right to
left. We end up with an expression of the type%
\[
=\left(  \sum\cdots\sum\right)  N\underset{n}{\underbrace{\partial
\partial\cdots\partial}}%
\]
with a single norm map and from each reduction step some finite sum remains.
It corresponds in each application of Proposition \ref{prop_NormCompat} to a
choice of the finitely many extensions of the valuation in question. It is
messy to spell this out since in each of our reduction steps the valuations
change, as we repeatedly switch to a finite field extension. Let us simply
write%
\[
=\sum_{\mathcal{F}(Y_{\bullet})}N\underset{n}{\underbrace{\partial
\partial\cdots\partial}}\text{,}%
\]
where%
\begin{equation}
\mathcal{F}(Y_{\bullet})=\left\{  (w_{n},\ldots,w_{1})\right\}  \label{luma8a}%
\end{equation}
is the finite set of all choices of valuations arising from running the above
procedure for the flag $Y_{\bullet}$. It is a little messy to unravel what the
elements of $\mathcal{F}(Y_{\bullet})$ are in a more explicit fashion (e.g.,
each $w_{i}$ is a rank one discrete valuation, but describing the field it is
defined on is already a little tricky). We will postpone describing
$\mathcal{F}(Y_{\bullet})$ in more detail.\newline(Step 4)\ Return to Equation
\ref{ltt1}. We obtain $\frac{1}{n!}\int_{X}c_{1}(\mathcal{O}_{X}(D))^{n}=$%
\begin{align*}
&  \frac{1}{n!}\sum_{Y_{\bullet}}[k\left(  Y_{n}\right)  :k]\\
&  \qquad\sum_{(w_{n},\ldots,w_{1})\in\mathcal{F}(Y_{\bullet})}N_{\kappa
(w_{n})/k\left(  Y_{n}\right)  }\partial_{w_{n}}\cdots\partial_{w_{2}}%
\partial_{w_{1}}\{f_{\alpha(Y_{0})\alpha(Y_{1})},f_{\alpha(Y_{1})\alpha
(Y_{2})},\ldots,f_{\alpha(Y_{n-1})\alpha(Y_{n})}\}
\end{align*}
and using Equation \ref{l_NormInTowers} as well as using that on
$K_{0}^{\operatorname*{M}}$ the norm is multiplication with the field
extension degree (see Equation \ref{lt2a}) once more, this simplifies to%
\begin{equation}
\frac{1}{n!}\sum_{Y_{\bullet}}\sum_{(w_{n},\ldots,w_{1})\in\mathcal{F}%
(Y_{\bullet})}[k\left(  w_{n}\right)  :k]\partial_{w_{n}}\cdots\partial
_{w_{1}}\{f_{\alpha(Y_{0})\alpha(Y_{1})},f_{\alpha(Y_{1})\alpha(Y_{2})}%
,\ldots,f_{\alpha(Y_{n-1})\alpha(Y_{n})}\}\text{.} \label{luma3aa}%
\end{equation}
We can rewrite this in terms of the Cartier divisor \v{C}ech representatives
$(h_{\alpha})_{\alpha}$ of Equation \ref{ltt0a}, giving%
\begin{equation}
\{f_{\alpha(Y_{0})\alpha(Y_{1})},\ldots,f_{\alpha(Y_{n-1})\alpha(Y_{n}%
)}\}=\sum_{c=0}^{n}(-1)^{c}\{h_{\alpha(Y_{0})},\ldots,\widehat{h_{\alpha
(Y_{c})}},\ldots,h_{\alpha(Y_{n})}\} \label{luma3a}%
\end{equation}
(we refer to Elaboration \ref{elab1} after the proof for details if this was
too quick). Thus,%
\[
=\frac{1}{n!}\sum_{Y_{\bullet}}\sum_{(w_{n},\ldots,w_{1})\in\mathcal{F}%
(Y_{\bullet})}\sum_{c=0}^{n}(-1)^{c}[k\left(  w_{n}\right)  :k]\partial
_{w_{n}}\cdots\partial_{w_{1}}\{h_{\alpha(Y_{0})},\ldots,\widehat
{h_{\alpha(Y_{c})}},\ldots,h_{\alpha(Y_{n})}\}
\]
Finally, by Proposition \ref{prop_DetFormula} and switching to the transpose
matrix (which does not affect the determinant), we obtain%
\begin{align*}
&  =\sum_{Y_{\bullet}}\sum_{(w_{n},\ldots,w_{1})\in\mathcal{F}(Y_{\bullet}%
)}\sum_{c=0}^{n}(-1)^{c}[k\left(  w_{n}\right)  :k]\\
&  \qquad\frac{1}{n!}\det%
\begin{pmatrix}
\underline{w}_{1}(h_{\alpha(Y_{0})}) & \cdots & \widehat{\underline{w}%
_{1}(h_{\alpha(Y_{c})})} & \cdots & \underline{w}_{1}(h_{\alpha(Y_{n})})\\
\underline{w}_{2}(h_{\alpha(Y_{0})}) & \ddots &  &  & \underline{w}%
_{2}(h_{\alpha(Y_{n})})\\
\vdots &  &  &  & \vdots\\
\underline{w}_{n}(h_{\alpha(Y_{0})}) & \cdots & \widehat{\underline{w}%
_{n}(h_{\alpha(Y_{c})})} & \cdots & \underline{w}_{n}(h_{\alpha(Y_{n})})
\end{pmatrix}
\text{,}%
\end{align*}
where $\underline{w}$ is the rank $n$ valuation coming from the $n$-DVF
structure induced from the valuations $w_{n},w_{n-1},\ldots,w_{1}$%
.\newline(Step 5) It remains to identify the mysterious set $\mathcal{F}%
(Y_{\bullet})$ from Equation \ref{luma8a}. We claim that $\mathcal{F}%
(Y_{\bullet})=\mathcal{G}(Y_{\bullet})$. Definition \ref{def_SetG} of
$\mathcal{G}(Y_{\bullet})$ is inductive in nature: It begins with choosing
$w_{1}$ and then going down to pick $w_{i}$ on the basis of $w_{i-1}$.
However, the choice of the valuations in $\mathcal{F}(Y_{\bullet})$ in the
above proof is also inductive in the same way. Hence, it suffices to show that
the inductive steps match (it is easy to see that the choice of $w_{1}$ is
done in the same way). To this end, we return to Equation \ref{luma8b}, where
we had defined that the next valuation $w_{i}$ \textquotedblleft runs through
the finitely many extensions of the valuation $v_{i_{p-1}}^{(p-1)}$ to the
finite field extension $\kappa(v_{i_{p}}^{(p)})$\textquotedblright. In this
step of the proof this is followed by one more map $\partial$ and one more map
$N$. Check that in the definition of $\partial$ we pick a valuation from the
maximal ideals of the integral closure\footnote{integral closure inside the
field of fractions of the domain} (see \S \ref{sect_BdyMapsBetweenPoints}
where this is carefully explained), and in the norm map, we also pick a
valuation from a maximal ideal of the integral closure\footnote{here it is the
integral closure in a finite extension of the field of fractions} (see
Proposition \ref{prop_NormCompat}, where this originates from). However,
instead of taking these two consecutive integral closures, we may right away
only do the second, giving the same outcome. We arrive at the same conditions
as in Definition \ref{def_SetG}.
\end{proof}

\begin{elab}
\label{elab1}We provide additional details for the computation in Equation
\ref{luma3a}. The exact sequence of Zariski sheaves $0\rightarrow
\mathcal{O}_{X}^{\times}\rightarrow\mathcal{K}_{X}^{\times}\rightarrow
\mathcal{K}_{X}^{\times}/\mathcal{O}_{X}^{\times}\rightarrow0$ induces the
connecting homomorphism $\delta$ in
\[
H^{0}(X,\mathcal{K}_{X}^{\times}/\mathcal{O}_{X}^{\times})\overset{\delta
}{\longrightarrow}H^{1}(X,\mathcal{O}_{X}^{\times})\text{.}%
\]
It sends $(h_{\alpha})_{\alpha}$ to $(f_{\alpha\beta})_{\alpha,\beta}$ with
$f_{\alpha\beta}:=h_{\beta}/h_{\alpha}^{-1}$, and if we use \v{C}ech
representatives on some open cover $\mathfrak{U}$ for $h$, we can still use
the same cover for $f$. We compute%
\begin{align*}
\{f_{\alpha(Y_{0})\alpha(Y_{1})},\ldots,f_{\alpha(Y_{n-1})\alpha(Y_{n})}\}  &
=\left\{  \frac{h_{\alpha(Y_{1})}}{h_{\alpha(Y_{0})}},\ldots,\frac
{h_{\alpha(Y_{n})}}{h_{\alpha(Y_{n-1})}}\right\} \\
&  =-\sum_{s_{1}=0,1}(-1)^{d_{1}}h_{\alpha(Y_{s_{1}})}\left\{  \frac
{h_{\alpha(Y_{2})}}{h_{\alpha(Y_{1})}},\ldots,\frac{h_{\alpha(Y_{n})}%
}{h_{\alpha(Y_{n-1})}}\right\}
\end{align*}
and inductively repeating the idea of the last term manipulation,%
\[
=(-1)^{n}\sum_{s_{1},\ldots,s_{n}\in\{0,1\}}(-1)^{d_{1}+\cdots+d_{n}%
}\{h_{\alpha(Y_{s_{1}})},h_{\alpha(Y_{s_{2}+1})},\ldots,h_{\alpha
(Y_{s_{n}+n-1})}\}\text{.}%
\]
As soon as two indices agree, the term is $2$-torsion by Lemma
\ref{lemma_GradedCommutative}, so cannot map to a non-trivial element in the
integers. Thus, only the selections of $n$ distinct consecutive elements
$s_{1},s_{2}+1,\ldots,s_{n}+n-1$ of $0<1<\cdots<n$ among these indices is
possibly non-zero. So it suffices if we only consider sequences of the shape
$(0,\ldots,0,1,1,\ldots,1)$ among the $s_{1},\ldots,s_{n}$. For $(1,1,\ldots
,1)$ we get the total sign $(-1)^{n}\cdot(-1)^{n}=+1$ and each time we
increase the number of leading zeros, the sign changes. Hence, we obtain%
\[
=\sum_{c=0}^{n}(-1)^{c}\{h_{\alpha(Y_{0})},\ldots,\widehat{h_{\alpha(Y_{c})}%
},\ldots,h_{\alpha(Y_{n})}\}
\]
as required.
\end{elab}

\begin{problem}
In our example in \S \ref{sect_Example} only two flags contribute. Under what
circumstances does only \emph{one} single flag $Y_{\bullet}$ contribute a
non-zero summand in the right side of the equation of Theorem
\ref{thm_MainFormula}? Can one give a general criterion? If this flag is
admissible, $\mathcal{G}(Y_{\bullet})$ contains only one element (Remark
\ref{rmk_GSetTrivialForAdmissibleFlag}) and one can choose $\underline{v}$ to
be this top rank valuation, too. Then the theorem gives an equation only
involving one single valuation on both sides of the equation. Can one
understand the equality in this case in terms of convex geometry?
\end{problem}

\subsection{Proof of the main result}

\begin{theorem}
\label{thm_B}Let $X/k$ be an irreducible smooth projective variety of
dimension $n$.

\begin{enumerate}
\item Let $\underline{v}:k\left(  X\right)  ^{\times}\rightarrow
\mathbb{Z}_{\operatorname*{lex}}^{n}$ be a top rank valuation.

\item Let $D$ be an ample divisor.

\item Suppose the graded semigroup $\Gamma_{\underline{v}}(D)$ is finitely generated.

\item Choose some $m\geq1$ sufficiently big so that $mD$ is very ample. Let
$(U_{\alpha},f_{\alpha})_{\alpha\in I}$ be the local trivialization of $mD$ as
a Cartier divisor on a finite open cover $(U_{\alpha})_{\alpha\in I}$ produced
by Lemma \ref{Lemma_Triv}.

\item Suppose%
\begin{equation}
X=\overset{\cdot}{\bigcup_{\alpha\in I}}\Sigma_{\alpha}\qquad\text{with}%
\qquad\Sigma_{\alpha}\subseteq U_{\alpha} \label{l_ThmB_DisjDecomp}%
\end{equation}
is a disjoint decomposition as a set.
\end{enumerate}

Then%
\begin{align*}
&  \operatorname*{Vol}\left(  \underset{\alpha\in I}{\mathrm{convex}\text{
}\mathrm{hull}}\left(  \frac{1}{m}\underline{v}(h_{\alpha})\right)  \right)
=\\
&  \qquad\qquad\qquad\sum_{Y_{\bullet}}\sum_{(w_{n},\ldots,w_{1}%
)\in\mathcal{G}(Y_{\bullet})}\sum_{c=0}^{n}(-1)^{c}[k\left(  w_{n}\right)
:k]\\
&  \qquad\qquad\qquad\qquad\qquad\operatorname*{Vol}\left.  \mathrm{simplex}%
\left\langle \underline{w}(h_{\alpha(Y_{0})}),\ldots\widehat{\underline
{w}(h_{\alpha(Y_{c})})}\ldots,\underline{w}(h_{\alpha(Y_{n})})\right\rangle
\right.
\end{align*}
and both values are the volume of the Newton--Okounkov body
$\operatorname*{Vol}\Delta_{Y_{\bullet}}(D)$. The first sum runs over all
flags $Y_{\bullet}$ in the sense of Definition \ref{def_Flag} and for each
flag $\mathcal{G}(Y_{\bullet})$ is as in Definition \ref{def_SetG}.
\end{theorem}

\begin{remark}
The right-hand side is independent of the choice of $\underline{v}$.
\end{remark}

We point out that by \textquotedblleft volume of the simplex\textquotedblright%
\ we mean the \textit{signed} volume, i.e. if the vectors are in the opposite
orientation to the standard basis of $\mathbb{R}^{n}$, the volume is
accordingly a negative value $-$ see Equations \ref{luma9a} and \ref{luma9}.

\begin{proof}
Use Theorem \ref{thm_MainFormula}. Then express $\operatorname*{Vol}%
\Delta_{\underline{v}}(D)$ using the presentation as a convex polytope coming
from Lemma \ref{Lemma_Triv}. We obtain%
\[
\operatorname*{Vol}\left(  \underset{\alpha\in I}{\mathrm{convex}\text{
}\mathrm{hull}}\left(  \frac{1}{m}\underline{v}(h_{\alpha})\right)  \right)
=\operatorname*{Vol}\Delta_{\underline{v}}(D)
\]
and this in turn equals%
\[
=\sum_{Y_{\bullet}}\sum_{(w_{n},\ldots,w_{1})\in\mathcal{F}(Y_{\bullet})}%
\sum_{c=0}^{n}(-1)^{c}[k\left(  w_{n}\right)  :k]\frac{1}{n!}\det\left(
\left[  \underline{w}_{\ell}(h_{\alpha(Y_{m})})\right]  _{\substack{\ell
=1,\ldots,n\\m=0,\ldots,\widehat{c},\ldots,n}}\right)  \text{,}%
\]
but for $n$ vectors $\underline{x}_{i}\in\mathbb{R}^{n}$, the expression
$\frac{1}{n!}\det\left(  \underline{x}_{1},\ldots,\underline{x}_{n}\right)  $
is the (signed) volume of the (oriented) $n$-simplex in $\mathbb{R}^{n}$
spanned by the vectors $\underline{x}_{1},\ldots,\underline{x}_{n}$. This is
clear: Without the factor $n!$ it is just the volume of the spanned parallelepiped.
\end{proof}

\section{A fully worked out example\label{sect_Example}}

In this section we will present a detailed example demonstrating our main
formula. Our example will depend on three parameters $l\in\mathbb{Z}_{\geq1}$,
$a,b\in\mathbb{Z}$, the first indicates running through an infinite family of
surfaces, while $a,b$ allow us to run through an infinite family of divisors
on them. Hence, in a sense, we discuss a countably infinite set of
examples.\medskip

Consider the Hirzebruch surface $F_{l}$. This is a toric surface, and this is
the viewpoint we shall exploit\footnote{One could also view $F_{l}$ as a
projective bundle over $\mathbb{P}^{1}$, but this perspective seems less
convenient when setting up the open covers of Theorem \ref{thm_B}.}. We will
use the notation of Fulton \cite{MR1234037}. Let $N:=\mathbb{Z}^{2}$. The
polyhedral fan $\Sigma$ in question is the following one.
\begin{equation}%
{\includegraphics[
height=1.257in,
width=1.1884in
]%
{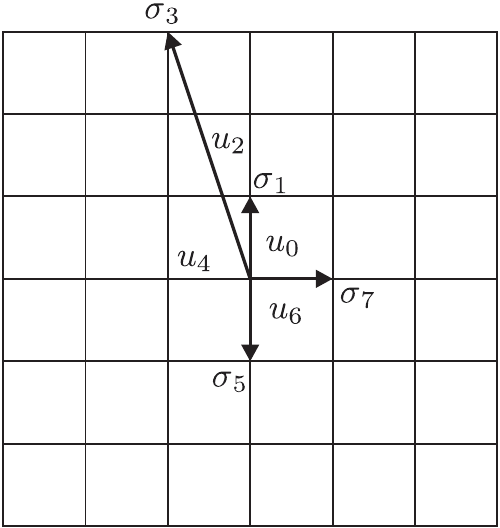}%
}%
\label{lafig1}%
\end{equation}
We write $u_{0},\ldots,u_{6}$ with even indices to denote the $2$-dimensional
cones, $\sigma_{1},\ldots,\sigma_{7}$ with odd indices to denote the
$1$-dimensional cones and $\eta$ will denote the origin. This unusual indexing
is used so that we have%
\begin{equation}
\text{the cone }u_{i}\text{ has the two facets }\sigma_{i-1}\text{ and }%
\sigma_{i+1} \label{lafig01}%
\end{equation}
for all $i$ (tacitly we set $\sigma_{-1}:=\sigma_{7}$). The ray $\sigma_{3}$
ends at the point $(-1,l)$. Write $X(\Sigma)$ for the toric variety attached
to a polyhedral fan $\Sigma$. Write $\Sigma\subseteq N_{\mathbb{R}}$ for the
polyhedral fan of Figure \ref{lafig1}, and in brief $X:=X(\Sigma)$. Our $X$ is
an integral smooth projective toric surface. We have $\operatorname*{Pic}%
(X)\cong\mathbb{Z}\left\langle \sigma_{1},\sigma_{3}\right\rangle $. Consider
the divisor%
\[
D=a\sigma_{1}+b\sigma_{3}\qquad\text{for}\qquad a,b\in\mathbb{Z}\text{.}%
\]
We claim that the divisor $D$ is ample if and only if%
\begin{equation}
a>0\qquad\text{and}\qquad b-la>0\text{.} \label{l_crit_ample_Example}%
\end{equation}
This is not too hard to see, but this computation is also carried out in
\cite[Example 6.1.16]{MR2810322}, use that in the notation loc. cit. our $D$
would be $aD_{4}+(b-la)D_{3}$, giving our claim here.

\subsection{Step 1: Local equations for the Cartier divisor}

We will be aiming towards Theorem \ref{thm_B}, so we need to fix a
trivialization for the line bundle $\mathcal{O}_{X}(D)$ in some open cover. To
this end, we shall use the standard affine opens $U_{i}:=X(u_{i})$ for $i\in
I:=\{0,2,4,6\}$ from the toric theory as our \v{C}ech cover $\mathfrak{U}$ of
$X$. This motivates why we denote the top-dimensional cones by the letter $u$.
Let us compute the Cartier divisor representative of $D\in H^{0}%
(X,\mathcal{K}_{X}^{\times}/\mathcal{O}_{X}^{\times})$. We write
$D=(h_{\alpha})_{\alpha\in I}$, and abstractly $D=\sum d_{i}\sigma_{i}$ for
odd $i$ and $d_{i}\in\mathbb{Z}$ (i.e. $d_{5}=d_{7}=0$). For each $i\in I$
this means that we need to solve the equations%
\begin{equation}
\left\langle h_{i},\sigma_{i-1}\right\rangle =-d_{i-1}\qquad\text{and}%
\qquad\left\langle h_{i},\sigma_{i+1}\right\rangle =-d_{i+1} \label{lima8c}%
\end{equation}
because $\sigma_{i-1}$, $\sigma_{i+1}$ are the facets of the cone $u_{i}$ by
our notational convention from Equation \ref{lafig01}, \cite[p. 61,
Lemma]{MR1234037}. Since the cone is smooth, these rays form a vector space
basis of $N_{\mathbb{R}}$. For example, for $i=2$ and if we write $x^{\xi
}y^{\psi}$ for the monomial exponents, we have to solve the equations%
\begin{align*}%
\begin{pmatrix}
\xi\\
\psi
\end{pmatrix}
\sigma_{1}  &  =-a\qquad\text{that is}\qquad%
\begin{pmatrix}
\xi\\
\psi
\end{pmatrix}
\cdot%
\begin{pmatrix}
0\\
1
\end{pmatrix}
=\psi=-a\\%
\begin{pmatrix}
\xi\\
\psi
\end{pmatrix}
\sigma_{3}  &  =-b\qquad\text{that is}\qquad%
\begin{pmatrix}
\xi\\
\psi
\end{pmatrix}
\cdot%
\begin{pmatrix}
-1\\
l
\end{pmatrix}
=-\xi+\psi l=-b
\end{align*}
and therefore $h_{2}:=x^{b-la}y^{-a}$. We leave the rest of this computation
(i.e. $h_{0},h_{4},h_{6}$) to the reader. The result is%
\begin{equation}
h_{0}:=y^{-a}\qquad h_{2}:=x^{b-la}y^{-a}\qquad h_{4}:=x^{b}\qquad
h_{6}:=1\text{.} \label{lima1}%
\end{equation}
Thus, the \v{C}ech $1$-cocycle for the line bundle class of $\mathcal{O}%
_{X}(D)$ in $\operatorname*{Pic}(X)\cong H^{1}(X,\mathcal{O}_{X}^{\times})$ is
given by $(f_{\alpha\beta})$ and $f_{\alpha\beta}:=h_{\beta}h_{\alpha}^{-1}$
(this corresponds to what we had discussed in Elaboration \ref{elab1}).

\begin{remark}
These computations also give the divisor polytope $P_{D}$ (as explained in
\cite[p. 66]{MR1234037}). Its defining inequalities are those of Equation
\ref{lima8c}, just replace \textquotedblleft$\left.  =\right.  $%
\textquotedblright\ with \textquotedblleft$\left.  \geq\right.  $%
\textquotedblright. If $D$ is ample (cf. Equation \ref{l_crit_ample_Example}),
we obtain the polytope%
\begin{equation}%
{\includegraphics[
height=1.2291in,
width=1.281in
]%
{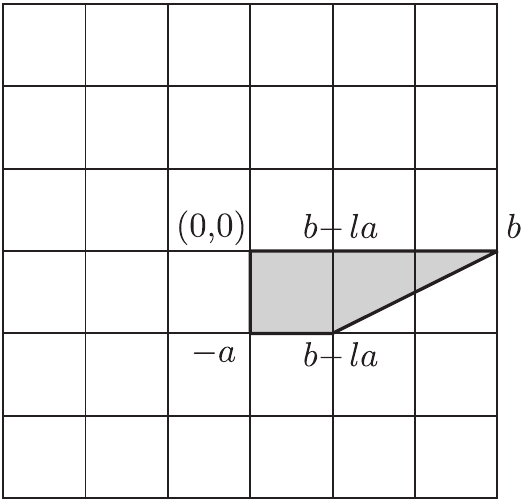}%
}%
\label{fig_A}%
\end{equation}
which has Euclidean volume $(b-la)a$ (for the box on the left) and $\frac
{1}{2}la^{2}$ (for the triangle on the right), and so in total $ba-\frac{1}%
{2}la^{2}$. By \cite[Proposition 6.1, (i)]{MR2571958} the Newton--Okounkov
body $\Delta_{Y_{\bullet}}(D)$ for a suitable flag (see loc. cit.) agrees with
$P_{D}$, so we learn that%
\begin{equation}
\frac{1}{2!}\operatorname*{Vol}\nolimits_{\operatorname*{birat}}%
(D)=\operatorname*{Vol}\nolimits_{\operatorname*{Eucl}}(P_{D})=ab-\frac{1}%
{2}la^{2} \label{lc1}%
\end{equation}
by \cite[Theorem A]{MR2571958}.
\end{remark}

\begin{example}
Let us double check the validity of this using intersection theory. We use the
quick formalism of \cite[\S 2.5, p. 43, p.44, last exercise]{MR1234037}. We
have%
\begin{align*}
l\cdot\sigma_{1}  &  =\sigma_{7}+\sigma_{3}\qquad\qquad0\cdot\sigma_{3}%
=\sigma_{1}+\sigma_{5}\\
(-l)\cdot\sigma_{5}  &  =\sigma_{3}+\sigma_{7}\qquad\qquad0\cdot\sigma
_{7}=\sigma_{1}+\sigma_{5}%
\end{align*}
and then $\sigma_{1}^{2}=-l$ and $\sigma_{3}^{2}=-0$ in the Chow ring, if we
indentify the names of the rays with their divisors. The divisors $\sigma
_{5},\sigma_{7}$ are linearly dependent in the Picard group. Thus,%
\begin{equation}
D^{2}=(a\sigma_{1}+b\sigma_{3})=a^{2}\sigma_{1}^{2}+2ab\sigma_{1}\sigma
_{3}+b^{2}\sigma_{3}^{2}=2ab-la^{2}\text{,} \label{lc1s}%
\end{equation}
which confirms Equation \ref{lc1} by using the characterization of the
birational volume through the asymptotic Riemann--Roch theorem. Again, this is
only valid since we have the assumptions of Equation
\ref{l_crit_ample_Example} in place.
\end{example}

\subsection{Step 2: Using the orbit decomposition}

In order to apply Theorem \ref{thm_B}, we need to cook up a disjoint
decomposition $(\Sigma_{j})_{j\in I}$ of $X$ as a set, as in Equation
\ref{l_ThmB_DisjDecomp}. We do this as follows: By the orbit-cone
correspondence the scheme $X$ can be written as the disjoint union of its
orbits $O(\tau)$, where $\tau$ runs through all cones of the polyhedral fan,
\cite[p. 54]{MR1234037} \cite[Theorem 3.2.6]{MR2810322}. Recall that,%
\begin{equation}
U_{\tau}=\coprod_{\mu\leq\tau}O(\mu)\text{,} \label{lima2}%
\end{equation}
so it will be easy to control $\Sigma_{j}\subseteq U_{j}$ for $j\in I$ if we
only build the $\Sigma_{j}$ from disjoint unions of orbits. Recall that $\eta$
is the trivial cone; just the origin. It corresponds to the big open orbit
$\mathbb{G}_{m}\times\mathbb{G}_{m}$ in the surface. Our naming convention
from Equation \ref{lafig01} suggests the following choice:
\begin{align}
\Sigma_{0}  &  :=O(\eta)\amalg O(\sigma_{-1})\amalg O(u_{0})\subseteq
U_{0}\nonumber\\
\Sigma_{2}  &  :=O(\sigma_{1})\amalg O(u_{2})\subseteq U_{2}\label{lafig2}\\
\Sigma_{4}  &  :=O(\sigma_{3})\amalg O(u_{4})\subseteq U_{4}\nonumber\\
\Sigma_{6}  &  :=O(\sigma_{5})\amalg O(u_{6})\subseteq U_{6}\text{.}\nonumber
\end{align}
Thus, except for adding the big open orbit $O(\eta)$ to $\Sigma_{0}$, we
always just take a rank one $\mathbb{G}_{m}$-torus $O(\sigma_{i-1})$ and a
single closed point $O(u_{i})$. All orbits are present, so this is a valid
disjoint decomposition.

\begin{remark}
The sets $\Sigma_{2}$, $\Sigma_{4}$ and $\Sigma_{4}$ are the underlying sets
of affine lines $\mathbb{A}_{k}^{1}$ in $X$. We will not use this.
\end{remark}

\subsection{Computing the right side}

\subsubsection{Reducing to a finite set of flags}

Let us compute the right side in Theorem \ref{thm_B}, that is%
\begin{align}
&  \sum_{Y_{\bullet}}\sum_{(w_{2},w_{1})\in\mathcal{G}(Y_{\bullet})}\sum
_{c=0}^{2}(-1)^{c}[k\left(  w_{2}\right)  :k]\label{lepsi1a}\\
&  \qquad\qquad\operatorname*{Vol}\left.  \mathrm{simplex}\left\langle
\underline{w}(h_{\alpha(Y_{0})}),\ldots\widehat{\underline{w}(h_{\alpha
(Y_{c})})}\ldots,\underline{w}(h_{\alpha(Y_{2})})\right\rangle \right.
\text{.} \label{lepsi1}%
\end{align}
A priori we have no control which of these uncountably many summands will be
non-zero, there is a huge supply of flags $Y_{\bullet}$. This is best
approached as follows: For the evaluation of each summand we only need to know
the values of $\alpha(Y_{i})\in I$ for $i=0,1,2$ and since $I:=\{0,2,4,6\}$ we
can make a case distinction depending on what values the $\alpha(Y_{i})$
attain. These a priori $4^{3}=64$ cases can quickly be cut down to a
manageable number:

\textit{Possible values for }$\alpha(Y_{0})$\textit{:} Since $Y_{0}$ is a
codimension zero integral closed subscheme of $X$, it can only be all of $X$.
The generic point of $X$ lies in $\Sigma_{0}$ by Equation \ref{lafig2}. Thus,
$\alpha(Y_{0})=0$ for \textit{all} flags.

\textit{Possible values for }$\alpha(Y_{1})$\textit{:} This is more
complicated. The variety $Y_{1}$ is a curve in $X$; we have such for all
possible values in $I$.

We use a trick: In Equation \ref{lepsi1} we only get a contribution if the
valuation vectors $\underline{w}(h_{\alpha(Y_{i})})$ are non-zero. Otherwise
the spanned $2$-simplex is degenerate and has volume zero. But all possible
values for $h_{\alpha(Y_{i})}$ are those listed in Equation \ref{lima1}. These
are (locally in the $U_{i}$) local equations for the divisors attached to the
rays. So, already when computing the first level valuation, i.e. the $v_{1}$
in%
\[
\underline{v}(f)=(\underline{v}_{1}(f),\underline{v}_{2}(f))\text{,}%
\]
it can only be non-zero at the divisors coming from these rays, i.e. the
$T$-invariant divisors on the toric surface. Similarly, all the closed points
which form the $Y_{2}$ of a flag, and which determine the second component
$\underline{v}_{2}(f)$, can only be possibly non-zero when they lie on such a
divisor. Thus, if at least one component of the vector $\underline{v}(f)$
needs to be non-zero, we can restrict for $Y_{1}$ to $T$-invariant divisors;
and indeed for $Y_{0}$ to the intersections of such divisors. Let us record
this essential simplification as a lemma:

\begin{lemma}
In Equation \ref{lepsi1a} a non-zero summand can only stem from a flag
$Y_{\bullet}$ such that%
\begin{equation}
Y_{0}=X\qquad Y_{1}=V(\sigma_{i})\qquad Y_{0}=V(u_{j})\text{,}\qquad
Y_{0}\supset Y_{1}\supset Y_{2}\text{,} \label{lima16a}%
\end{equation}
for some indices $i,j$ (here for an orbit $\tau$ we write $V(\tau)$ for the
orbit closure as in Fulton's book).
\end{lemma}

Besides cutting down the values of $Y_{\bullet}$ in the first sum in Equation
\ref{lepsi1a}, we get a few more useful facts from this: All the closed points
$V(u_{j})$ regardless the index $j$ are $k$-rational, so in Equation
\ref{lepsi1a} we have $[k\left(  w_{n}\right)  :k]=1$ for all flags we need to
consider. Furthermore, all these flags are flags of $T$-invariant divisors in
a smooth toric variety and therefore admissible, so the set $\mathcal{G}%
(Y_{\bullet})$ contains only a single element, Remark
\ref{rmk_GSetTrivialForAdmissibleFlag} (all the local rings in Definition
\ref{def_SetG} are already integrally closed). This also implies that the
valuation $\underline{w}$ is really just the one attached to $Y_{\bullet}$ itself.

\subsubsection{Eliminating a priori vanishing summands}

Our sum has therefore simplified to the finite sum%
\begin{equation}
\sum_{\substack{Y_{\bullet}\\\text{as in Eq. \ref{lima16a}}}}\sum_{c=0}%
^{2}(-1)^{c}\operatorname*{Vol}\left.  \mathrm{simplex}\left\langle
\underline{w}_{Y_{\bullet}}(h_{\alpha(Y_{0})}),\ldots\widehat{\underline
{w}_{Y_{\bullet}}(h_{\alpha(Y_{c})})}\ldots,\underline{w}_{Y_{\bullet}%
}(h_{\alpha(Y_{2})})\right\rangle \right.  \text{,} \label{lcips2}%
\end{equation}
where $\underline{w}_{Y_{\bullet}}$ is the valuation attached to $Y_{\bullet}%
$. The summation over $c$ triples the summands we have to evaluate. We can
simplify this: Inspecting Equation \ref{luma3aa} and Equation \ref{luma3a} in
the proof underlying our main formula, we see that we get a sum%
\[
\frac{1}{n!}\sum_{Y_{\bullet}}\sum_{(w_{n},\ldots,w_{1})\in\mathcal{F}%
(Y_{\bullet})}[k\left(  w_{n}\right)  :k]\partial_{w_{n}}\cdots\partial
_{w_{1}}\{f_{\alpha(Y_{0})\alpha(Y_{1})},f_{\alpha(Y_{1})\alpha(Y_{2})}%
,\ldots,f_{\alpha(Y_{n-1})\alpha(Y_{n})}\}
\]
first $-$ which later gets transformed into a sum of the shape as in Equation
\ref{lcips2} (see Elaboration \ref{elab1} for details). In our case, the inner
term is%
\[
\frac{1}{2!}\sum_{Y_{\bullet}}\partial_{w_{2}}\partial_{w_{1}}\{f_{\alpha
(Y_{0})\alpha(Y_{1})},f_{\alpha(Y_{1})\alpha(Y_{2})}\}\text{,}%
\]
where $(f_{\alpha\beta})$ are the local equations for the $1$-cocycle of
$\mathcal{O}_{X}(D)$, exactly as explained below Equation \ref{lima1}. We use
the following trick:\ If for a flag $Y_{\bullet}$ the summand%
\[
\partial_{w_{2}}\partial_{w_{1}}\{f_{\alpha(Y_{0})\alpha(Y_{1})}%
,f_{\alpha(Y_{1})\alpha(Y_{2})}\}
\]
is zero, then upon the rewriting in Elaboration \ref{elab1} the $3$ resulting
terms necessarily add up to zero, too, so there is no harm in neglecting these
flags as summands altogether.

Let us carry out this computation now:\ We use the indices $i$ and $j$ as
introduced in Equation \ref{lima16a}. Letting columns represent the values of
$i$ and rows the values of $j$, we obtain the possible summands; their
arguments are%
\[%
\begin{tabular}
[c]{c|cccc}
& $1$ & $3$ & $5$ & $7$\\\hline
$0$ & $\{f_{02},f_{20}\}$ & $\{f_{04},f_{40}\}$ & $\{f_{06},f_{60}\}$ &
$\{f_{00},f_{00}\}$\\
$2$ & $\{f_{02},f_{22}\}$ & $\{f_{04},f_{42}\}$ & $\{f_{06},f_{62}\}$ &
$\{f_{00},f_{02}\}$\\
$4$ & $\{f_{02},f_{24}\}$ & $\{f_{04},f_{44}\}$ & $\{f_{06},f_{64}\}$ &
$\{f_{00},f_{04}\}$\\
$6$ & $\{f_{02},f_{26}\}$ & $\{f_{04},f_{46}\}$ & $\{f_{06},f_{66}\}$ &
$\{f_{00},f_{06}\}$%
\end{tabular}
\ \ \ \ \
\]
Most terms must be zero on general grounds: If both indices agree, we have
$f_{\alpha\alpha}=1$ because%
\begin{equation}
f_{\alpha\beta}:=h_{\beta}h_{\alpha}^{-1}\text{.} \label{l_linz}%
\end{equation}
This kills for example the entire right column and a downward diagonal. By the
same formula, we also have $\{f_{\alpha\beta},f_{\beta\alpha}\}=\{f_{\alpha
\beta},f_{\alpha\beta}^{-1}\}$ and for any $f$, $\{f,f^{-1}\}$ is $2$-torsion
in $K_{2}^{\operatorname*{M}}$ by\ Lemma \ref{lemma_GradedCommutative}, so
must be mapped to zero under any map to the reals, and therefore cannot
contribute non-trivially. This leaves only the following pairs $(i,j)$,%
\[%
\begin{tabular}
[c]{c|cccc}
& $1$ & $3$ & $5$ & $7$\\\hline
$0$ &  &  &  & \\
$2$ &  & $\{f_{04},f_{42}\}$ & $\{f_{06},f_{62}\}$ & \\
$4$ & $\{f_{02},f_{24}\}$ &  & $\{f_{06},f_{64}\}$ & \\
$6$ & $\{f_{02},f_{26}\}$ & $\{f_{04},f_{46}\}$ &  &
\end{tabular}
\ \ \ \ \
\]
However, many of these index pairs $(i,j)$ cannot occur such that $Y_{\bullet
}$ is a flag, i.e. the inclusion condition in Equation \ref{lima16a} would be
broken. In detail: Note that $\sigma_{i}$ according to Equation \ref{lafig2}
lies in $\Sigma_{i+1}$, so $\alpha(\sigma_{i})=i+1$ (and read $\Sigma_{8}$ as
$\Sigma_{0}$ in the case of $\sigma_{7}$). Similarly, the single point of the
orbit $u_{j}$ lies in $\Sigma_{j}$. Once we are in the orbit $O(\sigma_{i})$,
then by the orbit-cone correspondence, its closure is%
\[
V(\sigma_{i})=O(\sigma_{i})\amalg O(u_{i-1})\amalg O(u_{i+1})\text{,}%
\]
by \cite[p. 54]{MR1234037}. Again, our special indexing comes in handy. We
deduce that if $Y_{1}=V(\sigma_{i})$, then the closed point $Y_{2}$ of a flag
can only lie in one of these three sets.

If it lies on the orbit $O(\sigma_{i})$, this means $\alpha(Y_{1}%
)=\alpha(Y_{2})$ and thus $f_{\alpha(Y_{1})\alpha(Y_{2})}=1$ by Equation
\ref{l_linz}, implying that the term $\{-,-\}$ is zero. Hence, we may restrict
to $Y_{2}$ being the closed points $O(u_{i-1})$ or $O(u_{i+1})$. Again, by
Equation \ref{lafig2} if it is $O(u_{i+1})$, they still both lie in
$\Sigma_{i+1}$, so still $\alpha(Y_{1})=\alpha(Y_{2})$, and thus $\{-,-\}$ is
zero. Thus, only $O(u_{i-1})$ is possible if we want a non-zero contribution.
Using this constraint, only two possibly non-zero summands remain:%
\[%
\begin{tabular}
[c]{c|cccc}
& $1$ & $3$ & $5$ & $7$\\\hline
$0$ &  &  &  & \\
$2$ &  & $\{f_{04},f_{42}\}$ &  & \\
$4$ &  &  & $\{f_{06},f_{64}\}$ & \\
$6$ &  &  &  &
\end{tabular}
\ \ \
\]
and these belong to the following flags:

\begin{enumerate}
\item $Y_{\bullet}=(X\supset V(\sigma_{3})\supset V(u_{2}))$ with
$\alpha(Y_{0})=0$, $\alpha(Y_{1})=4$, $\alpha(Y_{2})=2$, and

\item $Y_{\bullet}=(X\supset V(\sigma_{5})\supset V(u_{4}))$ with
$\alpha(Y_{0})=0$, $\alpha(Y_{1})=6$, $\alpha(Y_{2})=4$.
\end{enumerate}

The whole point of the paper is the presence of different flags along which we
take the span of valuation vectors, and here we have isolated \textit{the} two
critical flags in our example. Next, we return to our original formula in
Equation \ref{lcips2} $-$ knowing that we only need to evaluate it for these
two flags $Y_{\bullet}$.

\subsubsection{The simplex attached to $X\supset V(\sigma_{3})\supset
V(u_{2})$}

For the flag $X\supset V(\sigma_{3})\supset V(u_{2})$, we can conveniently
compute the valuation vectors in an affine open of $X$. The closed point
$V(u_{2})$ lies in $U_{2}$ by the orbit-cone correspondence. We easily compute
that $U_{2}=\operatorname*{Spec}k[x^{-1},x^{l}y]$. In this open the local
equation for $V(\sigma_{3})$ is the principal ideal $(x^{-1})$, so this
corresponds to the valuation $\underline{w}_{1}$, and $V(u_{2})$ is cut out by
$(x^{-1},x^{l}y)$, so $x^{l}y$ is a uniformizer for the component
$\underline{w}_{2}$. Equation \ref{lcips2} asks us to compute%
\begin{align*}
&  \sum_{c=0}^{2}(-1)^{c}\operatorname*{Vol}\left.  \mathrm{simplex}%
\left\langle \underline{w}_{Y_{\bullet}}(h_{\alpha(Y_{0})}),\ldots
\widehat{\underline{w}_{Y_{\bullet}}(h_{\alpha(Y_{c})})}\ldots,\underline
{w}_{Y_{\bullet}}(h_{\alpha(Y_{2})})\right\rangle \right. \\
&  \qquad=\operatorname*{Vol}\left.  \mathrm{simplex}\left\langle
\underline{w}_{Y_{\bullet}}(h_{4}),\underline{w}_{Y_{\bullet}}(h_{2}%
)\right\rangle \right.  -\operatorname*{Vol}\left.  \mathrm{simplex}%
\left\langle \underline{w}_{Y_{\bullet}}(h_{0}),\underline{w}_{Y_{\bullet}%
}(h_{2})\right\rangle \right. \\
&  \qquad\qquad\qquad+\operatorname*{Vol}\left.  \mathrm{simplex}\left\langle
\underline{w}_{Y_{\bullet}}(h_{0}),\underline{w}_{Y_{\bullet}}(h_{4}%
)\right\rangle \right.  \text{.}%
\end{align*}
We get%
\begin{align*}
&  \operatorname*{Vol}\left.  \mathrm{simplex}\left\langle \underline
{w}_{Y_{\bullet}}(h_{4}),\underline{w}_{Y_{\bullet}}(h_{2})\right\rangle
\right.  =\frac{1}{2}\det%
\begin{pmatrix}
\underline{w}_{1}(x^{b}) & \underline{w}_{1}(x^{b-la}y^{-a})\\
\underline{w}_{2}(x^{b}) & \underline{w}_{2}(x^{b-la}y^{-a})
\end{pmatrix}
\\
&  \qquad=\frac{1}{2}\det%
\begin{pmatrix}
\underline{w}_{1}(x^{b}) & \underline{w}_{1}((x^{l}y)^{-a}\cdot x^{b})\\
\underline{w}_{2}(x^{b}) & \underline{w}_{2}((x^{l}y)^{-a}\cdot x^{b})
\end{pmatrix}
=\frac{1}{2}\det%
\begin{pmatrix}
-b & -b\\
0 & -a
\end{pmatrix}
=\frac{1}{2}ab\text{.}%
\end{align*}
The next $2$-simplex is more interesting. Note that $y^{-a}=\left(
x^{l}y\right)  ^{-a}\cdot x^{la}$ and therefore%
\begin{align*}
&  \operatorname*{Vol}\left.  \mathrm{simplex}\left\langle \underline
{w}_{Y_{\bullet}}(h_{0}),\underline{w}_{Y_{\bullet}}(h_{2})\right\rangle
\right.  =\frac{1}{2}\det%
\begin{pmatrix}
\underline{w}_{1}(\left(  x^{l}y\right)  ^{-a}\cdot x^{la}) & \underline
{w}_{1}((x^{l}y)^{-a}\cdot x^{b})\\
\underline{w}_{2}(\left(  x^{l}y\right)  ^{-a}\cdot x^{la}) & \underline
{w}_{2}((x^{l}y)^{-a}\cdot x^{b})
\end{pmatrix}
\\
&  \qquad=\frac{1}{2}\det%
\begin{pmatrix}
-la & -b\\
-a & -a
\end{pmatrix}
=\frac{1}{2}\left(  la^{2}-ab\right)  \text{,}%
\end{align*}
and finally, $\operatorname*{Vol}\left.  \mathrm{simplex}\left\langle
\underline{w}_{Y_{\bullet}}(h_{0}),\underline{w}_{Y_{\bullet}}(h_{4}%
)\right\rangle \right.  $ equals%
\[
=\frac{1}{2}\det%
\begin{pmatrix}
\underline{w}_{1}(\left(  x^{l}y\right)  ^{-a}\cdot x^{la}) & \underline
{w}_{1}(x^{b})\\
\underline{w}_{2}(\left(  x^{l}y\right)  ^{-a}\cdot x^{la}) & \underline
{w}_{2}(x^{b})
\end{pmatrix}
=\frac{1}{2}\det%
\begin{pmatrix}
-la & -b\\
-a & 0
\end{pmatrix}
=-\frac{1}{2}ab\text{.}%
\]
Thus, the total summand for this flag in Equation \ref{lcips2} is%
\[
\frac{1}{2}ab-\frac{1}{2}\left(  la^{2}-ab\right)  +\frac{1}{2}\left(
-ab\right)  =\frac{1}{2}ab-\frac{1}{2}la^{2}\text{.}%
\]

\subsubsection{The simplex attached to $X\supset V(\sigma_{5})\supset
V(u_{4})$}

Now, we need to do the analogous evaluation for this flag, but this turns out
to be shorter. We write $\underline{w}^{\prime}$ for the valuation of this
flag. The affine open $U_{4}=\operatorname*{Spec}k[x^{-1},x^{-l}y^{-1}]$
contains $O(u_{4})$ and $O(\sigma_{5})$ by Equation \ref{lima2}. The local
equation for the orbit closure $V(\sigma_{5})$ is the principal ideal
$(x^{-1})$. The rank one valuation along this ideal is the first component
$\underline{w}_{1}^{\prime}$. Moreover, $x^{-l}y^{-1}$ is a uniformizer for
the valuation $\underline{w}_{2}^{\prime}$. For this flag, Equation
\ref{lcips2} adds the contribution from%
\begin{align*}
&  \sum_{c=0}^{2}(-1)^{c}\operatorname*{Vol}\left.  \mathrm{simplex}%
\left\langle \underline{w}_{Y_{\bullet}}(h_{\alpha(Y_{0})}),\ldots
\widehat{\underline{w}_{Y_{\bullet}}(h_{\alpha(Y_{c})})}\ldots,\underline
{w}_{Y_{\bullet}}(h_{\alpha(Y_{2})})\right\rangle \right. \\
&  \qquad=\operatorname*{Vol}\left.  \mathrm{simplex}\left\langle
\underline{w}_{Y_{\bullet}}(h_{6}),\underline{w}_{Y_{\bullet}}(h_{4}%
)\right\rangle \right.  -\operatorname*{Vol}\left.  \mathrm{simplex}%
\left\langle \underline{w}_{Y_{\bullet}}(h_{0}),\underline{w}_{Y_{\bullet}%
}(h_{4})\right\rangle \right. \\
&  \qquad\qquad\qquad+\operatorname*{Vol}\left.  \mathrm{simplex}\left\langle
\underline{w}_{Y_{\bullet}}(h_{0}),\underline{w}_{Y_{\bullet}}(h_{6}%
)\right\rangle \right.  \text{.}%
\end{align*}
Note that $h_{6}=1$, so $\underline{w}_{Y_{\bullet}}(h_{6})=0$ irrespective
the flag. Thus, two of these $2$-simplices are degenerate and contribute no
volume. The only remaining term is%
\begin{align*}
&  =-\operatorname*{Vol}\left.  \mathrm{simplex}\left\langle \underline
{w}_{Y_{\bullet}}(h_{0}),\underline{w}_{Y_{\bullet}}(h_{4})\right\rangle
\right.  =-\frac{1}{2}\det%
\begin{pmatrix}
\underline{w}_{1}^{\prime}(y^{-a}) & \underline{w}_{1}^{\prime}(x^{b})\\
\underline{w}_{2}^{\prime}(y^{-a}) & \underline{w}_{2}^{\prime}(x^{b})
\end{pmatrix}
\\
&  =-\frac{1}{2}\det%
\begin{pmatrix}
\underline{w}_{1}^{\prime}((x^{-l}y^{-1})^{a}\cdot x^{la}) & \underline{w}%
_{1}^{\prime}(x^{b})\\
\underline{w}_{2}^{\prime}((x^{-l}y^{-1})^{a}\cdot x^{la}) & \underline{w}%
_{2}^{\prime}(x^{b})
\end{pmatrix}
=-\frac{1}{2}\det%
\begin{pmatrix}
-la & -b\\
a & 0
\end{pmatrix}
=\frac{1}{2}ab\text{.}%
\end{align*}

\subsection{Conclusion}

Combining the contribution of both flags, we obtain that the right-hand side
of Theorem \ref{thm_B} is%
\[
\left(  \frac{1}{2}ab-\frac{1}{2}la^{2}\right)  +\left(  \frac{1}{2}ab\right)
=ab-\frac{1}{2}la^{2}\text{.}%
\]
We had already computed the volume of the Newton--Okounkov body in two
different ways in Equation \ref{lc1} as well as Equation \ref{lc1s}, and
(luckily!) we see that all values agree. It remains to compute the left-hand
side of Theorem \ref{thm_B}. Recall that this side originates from Proposition
\ref{prop_TrivPolytope}. We need to pick $m$ such that $mD$ is very ample, but
on a toric variety a line bundle is ample if and only if it is very ample, so
$m=1$ is enough. But then note that the trivializing equations are those
listed in Equation \ref{lima1}, but they also mark the extremal points of the
polytope in Figure \ref{fig_A}, so this \textit{is} the convex polytope in
question. In total, Theorem \ref{thm_B} hence reads: For the standard
lexicographic valuation $\underline{v}$ of a toric variety we have%
\begin{align*}
&  \operatorname*{Vol}\left(  \mathrm{convex}\text{ }\mathrm{hull}\left(
\underline{v}(y^{-a});\underline{v}(x^{b-la}y^{-a});\underline{v}%
(x^{b});0\right)  \right)  =\\
&  \qquad\frac{1}{2}\det%
\begin{pmatrix}
\underline{w}_{1}(x^{b}) & \underline{w}_{1}(x^{b-la}y^{-a})\\
\underline{w}_{2}(x^{b}) & \underline{w}_{2}(x^{b-la}y^{-a})
\end{pmatrix}
-\frac{1}{2}\det%
\begin{pmatrix}
\underline{w}_{1}(\left(  x^{l}y\right)  ^{-a}\cdot x^{la}) & \underline
{w}_{1}((x^{l}y)^{-a}\cdot x^{b})\\
\underline{w}_{2}(\left(  x^{l}y\right)  ^{-a}\cdot x^{la}) & \underline
{w}_{2}((x^{l}y)^{-a}\cdot x^{b})
\end{pmatrix}
\\
&  \qquad+\frac{1}{2}\det%
\begin{pmatrix}
\underline{w}_{1}(\left(  x^{l}y\right)  ^{-a}\cdot x^{la}) & \underline
{w}_{1}(x^{b})\\
\underline{w}_{2}(\left(  x^{l}y\right)  ^{-a}\cdot x^{la}) & \underline
{w}_{2}(x^{b})
\end{pmatrix}
-\frac{1}{2}\det%
\begin{pmatrix}
\underline{w}_{1}^{\prime}(y^{-a}) & \underline{w}_{1}^{\prime}(x^{b})\\
\underline{w}_{2}^{\prime}(y^{-a}) & \underline{w}_{2}^{\prime}(x^{b})
\end{pmatrix}
\text{.}%
\end{align*}
This agrees with Equation \ref{lvj}, except that all the arguments are spelt out.%

\appendix

\section{Alternative proof of the main theorem}

In this appendix, we sketch an alternative proof of the main results. It
removes any use of Milnor $K$-theory and uses residues instead. The proof is a
little weaker because it only works in characteristic zero.

We only describe how to change the proof. Firstly, in Diagram \ref{lfe1}
replace $H^{n}(X,\mathcal{K}_{n}^{\operatorname*{M}})$ by $H^{n}%
(X,\Omega_{X/\mathbb{C}}^{n})$ and $\gamma$ by the differential logarithm%
\begin{align*}
H^{\ast}(X,\mathcal{O}_{X}^{\times}\otimes\cdots\otimes\mathcal{O}_{X}%
^{\times})  &  \longrightarrow H^{\ast}(X,\mathcal{K}_{\ast}%
^{\operatorname*{M}})\longrightarrow H^{\ast}(X,\Omega_{X/k}^{\ast})\\
x_{1}\otimes\cdots\otimes x_{n}  &  \longmapsto\{x_{1},\ldots,x_{n}%
\}\longmapsto\frac{\mathrm{d}x_{1}}{x_{1}}\wedge\cdots\wedge\frac
{\mathrm{d}x_{n}}{x_{n}}\text{.}%
\end{align*}
If we work over the complex numbers, the map $\tau$ of Diagram \ref{lfe1} then
can be replaced by%
\[
H^{n}(X,\Omega_{X/\mathbb{C}}^{n})\cong H^{n,n}(X)\subseteq H^{2n}%
(X,\mathbb{C})\longrightarrow\mathbb{C}\text{,}%
\]
i.e. under the Dolbeault isomorphism we identify the cohomology group
$H^{n}(X,\Omega_{X/\mathbb{C}}^{n})$ with the $(n,n)$-classes in the top Betti
cohomology. Doing these replacements is compatible with intersection theory,
i.e. this still computes the same intersection number. However, even if $k$ is
not the complex numbers, we can also just use the evaluation map of Serre
duality, $H^{n}(X,\Omega_{X/k}^{n})\rightarrow k$, using that $X/k$ is
integral smooth proper over $k$. Now use the Cousin resolution of
$\Omega_{X/k}^{n}$, \cite{MR0222093}. Its terms are skyscraper sheaves of the
shape%
\[
U\mapsto\bigoplus_{x\in U^{(i)}}H_{x}^{p}(X,\Omega_{X/k}^{n})\text{,}%
\qquad\text{where }H_{x}^{p}(-,-)\text{ denotes local cohomology,}%
\]
where $U^{(i)}$ denotes the set of points $x$ such that $\overline{\{x\}}$ has
codimension $i$ in $X$. The `algebraic partitions of unity' of \cite[\S 2.1]%
{MR3104562} exist for such sheaves. Notably, \cite[Lemma 2]{MR3104562}
applies. Then form the double complex comparing the \v{C}ech cohomology of
$\Omega_{X/k}^{n}$ in a concrete open cover $\mathfrak{U}$ with the cohomology
of the Cousin complex, imitating \cite[\S 2.2]{MR3104562}. The boundary maps
$\partial_{(-)}$ get replaced by residue maps; the norm maps get replaced by
trace maps. One proves the analogue of Proposition \ref{prop_DetFormula},
where the left side gets replaced by $\operatorname*{res}_{v_{n}}%
\cdots\operatorname*{res}_{v_{1}}\left(  \frac{\mathrm{d}f_{1}}{f_{1}}%
\wedge\cdots\wedge\frac{\mathrm{d}f_{n}}{f_{n}}\right)  $ and
\textquotedblleft$\operatorname*{res}_{v}$\textquotedblright\ being the local
residue map. This term can be understood as a Grothendieck residue symbol.
Once this is all set up, adapt the proof of Theorem \ref{thm_MainFormula}. If
one runs this variant of the proof in chacteristic $p>0$, we only obtain
equality in Theorem \ref{thm_MainFormula} modulo $p$, so this proof is
strictly weaker than the one using Milnor $K$-groups.

\bibliographystyle{amsalpha}
\bibliography{ollinewbib}

\newcommand{\etalchar}[1]{$^{#1}$}
\def\cprime{$'$} \def\polhk#1{\setbox0=\hbox{#1}{\ooalign{\hidewidth
  \lower1.5ex\hbox{`}\hidewidth\crcr\unhbox0}}} \def\cprime{$'$}
  \def\polhk#1{\setbox0=\hbox{#1}{\ooalign{\hidewidth
  \lower1.5ex\hbox{`}\hidewidth\crcr\unhbox0}}} \def\cprime{$'$}
  \def\cprime{$'$} \def\cprime{$'$} \def\cprime{$'$}
\providecommand{\bysame}{\leavevmode\hbox to3em{\hrulefill}\thinspace}
\providecommand{\MR}{\relax\ifhmode\unskip\space\fi MR }
\providecommand{\MRhref}[2]{%
  \href{http://www.ams.org/mathscinet-getitem?mr=#1}{#2}
}
\providecommand{\href}[2]{#2}
\begin{thebibliography}{CFK{\etalchar{+}}17}

\bibitem[AKL14]{MR3207370}
D.~Anderson, A.~K\"{u}ronya, and V.~Lozovanu, \emph{Okounkov bodies of finitely
  generated divisors}, Int. Math. Res. Not. IMRN (2014), no.~9, 2343--2355.
  \MR{3207370}

\bibitem[And13]{MR3063911}
D.~Anderson, \emph{Okounkov bodies and toric degenerations}, Math. Ann.
  \textbf{356} (2013), no.~3, 1183--1202. \MR{3063911}

\bibitem[Blo74]{MR0342514}
S.~Bloch, \emph{{$K{\rm _{2}}$} and algebraic cycles}, Ann. of Math. (2)
  \textbf{99} (1974), 349--379. \MR{0342514 (49 \#7260)}

\bibitem[Bra13]{MR3104562}
O.~Braunling, \emph{Bad intersections and constructive aspects of the
  {B}loch-{Q}uillen formula}, New York J. Math. \textbf{19} (2013), 545--564.
  \MR{3104562}

\bibitem[CFK{\etalchar{+}}17]{MR3694973}
C.~Ciliberto, M.~Farnik, A.~K\"{u}ronya, V.~Lozovanu, J.~Ro\'{e}, and
  C.~Shramov, \emph{Newton-{O}kounkov bodies sprouting on the valuative tree},
  Rend. Circ. Mat. Palermo (2) \textbf{66} (2017), no.~2, 161--194.
  \MR{3694973}

\bibitem[CLS11]{MR2810322}
D.~Cox, J.~Little, and H.~Schenck, \emph{Toric varieties}, Graduate Studies in
  Mathematics, vol. 124, American Mathematical Society, Providence, RI, 2011.
  \MR{2810322 (2012g:14094)}

\bibitem[Ful93]{MR1234037}
W.~Fulton, \emph{Introduction to toric varieties}, Annals of Mathematics
  Studies, vol. 131, Princeton University Press, Princeton, NJ, 1993, The
  William H. Roever Lectures in Geometry. \MR{1234037 (94g:14028)}

\bibitem[GH94]{MR1288523}
P.~Griffiths and J.~Harris, \emph{Principles of algebraic geometry}, Wiley
  Classics Library, John Wiley \& Sons, Inc., New York, 1994, Reprint of the
  1978 original. \MR{1288523}

\bibitem[Gil05]{MR2181825}
H.~Gillet, \emph{{$K$}-theory and intersection theory}, Handbook of
  {$K$}-theory. {V}ol. 1, 2, Springer, Berlin, 2005, pp.~235--293. \MR{2181825
  (2006h:14013)}

\bibitem[Gra78]{MR0491685}
D.~Grayson, \emph{Products in {$K$}-theory and intersecting algebraic cycles},
  Invent. Math. \textbf{47} (1978), no.~1, 71--83. \MR{0491685}

\bibitem[GS06]{MR2266528}
P.~Gille and T.~Szamuely, \emph{Central simple algebras and {G}alois
  cohomology}, Cambridge Studies in Advanced Mathematics, vol. 101, Cambridge
  University Press, Cambridge, 2006. \MR{2266528 (2007k:16033)}

\bibitem[Har66]{MR0222093}
R.~Hartshorne, \emph{Residues and duality}, Lecture notes of a seminar on the
  work of A. Grothendieck, given at Harvard 1963/64. With an appendix by P.
  Deligne. Lecture Notes in Mathematics, No. 20, Springer-Verlag, Berlin, 1966.
  \MR{0222093 (36 \#5145)}

\bibitem[Ker09]{MR2461425}
M.~Kerz, \emph{The {G}ersten conjecture for {M}ilnor {$K$}-theory}, Invent.
  Math. \textbf{175} (2009), no.~1, 1--33. \MR{2461425 (2010i:19004)}

\bibitem[KK12]{MR2950767}
K.~Kaveh and A.~G. Khovanskii, \emph{Newton-{O}kounkov bodies, semigroups of
  integral points, graded algebras and intersection theory}, Ann. of Math. (2)
  \textbf{176} (2012), no.~2, 925--978. \MR{2950767}

\bibitem[KL18]{MR3781431}
A.~K\"{u}ronya and V.~Lozovanu, \emph{Local positivity of linear series on
  surfaces}, Algebra Number Theory \textbf{12} (2018), no.~1, 1--34.
  \MR{3781431}

\bibitem[KLM12]{MR2889138}
A.~K\"{u}ronya, V.~Lozovanu, and C.~Maclean, \emph{Convex bodies appearing as
  {O}kounkov bodies of divisors}, Adv. Math. \textbf{229} (2012), no.~5,
  2622--2639. \MR{2889138}

\bibitem[Laz04]{MR2095471}
R.~Lazarsfeld, \emph{Positivity in algebraic geometry. {I}}, Ergebnisse der
  Mathematik und ihrer Grenzgebiete. 3. Folge. A Series of Modern Surveys in
  Mathematics [Results in Mathematics and Related Areas. 3rd Series. A Series
  of Modern Surveys in Mathematics], vol.~48, Springer-Verlag, Berlin, 2004,
  Classical setting: line bundles and linear series. \MR{2095471}

\bibitem[LM09]{MR2571958}
R.~Lazarsfeld and M.~Musta\c{t}\u{a}, \emph{Convex bodies associated to linear
  series}, Ann. Sci. \'{E}c. Norm. Sup\'{e}r. (4) \textbf{42} (2009), no.~5,
  783--835. \MR{2571958}

\bibitem[Mil70]{MR0260844}
J.~Milnor, \emph{Algebraic {$K$}-theory and quadratic forms}, Invent. Math.
  \textbf{9} (1969/1970), 318--344. \MR{0260844 (41 \#5465)}

\bibitem[Oko96]{MR1400312}
A.~Okounkov, \emph{Brunn-{M}inkowski inequality for multiplicities}, Invent.
  Math. \textbf{125} (1996), no.~3, 405--411. \MR{1400312}

\bibitem[Oko03]{MR1995384}
\bysame, \emph{Why would multiplicities be log-concave?}, The orbit method in
  geometry and physics ({M}arseille, 2000), Progr. Math., vol. 213,
  Birkh\"{a}user Boston, Boston, MA, 2003, pp.~329--347. \MR{1995384}

\bibitem[Par83]{MR697316}
A.~N. Parshin, \emph{Chern classes, ad\`eles and {$L$}-functions}, J. Reine
  Angew. Math. \textbf{341} (1983), 174--192. \MR{697316 (85c:14015)}

\bibitem[PF99]{ParshinAdeleTheory}
A.~Parshin and T.~Fimmel, \emph{{Introduction to higher adelic theory
  (draft)}}, {unpublished}, 1999.

\bibitem[Ros96]{MR1418952}
M.~Rost, \emph{Chow groups with coefficients}, Doc. Math. \textbf{1} (1996),
  No. 16, 319--393 (electronic). \MR{1418952 (98a:14006)}

\bibitem[ZS75]{MR0389876}
O.~Zariski and P.~Samuel, \emph{Commutative algebra. {V}ol. {II}},
  Springer-Verlag, New York-Heidelberg, 1975, Reprint of the 1960 edition,
  Graduate Texts in Mathematics, Vol. 29. \MR{0389876}

\end{thebibliography}

\end{document}